\documentclass[11pt]{article}
\usepackage{amsmath}
\usepackage{algorithm}
\usepackage{algpseudocode}

\usepackage[english]{babel}

\usepackage[letterpaper,top=2cm,bottom=2cm,left=3cm,right=3cm,marginparwidth=1.75cm]{geometry}

\usepackage{adjustbox}
\usepackage[english]{babel}
\usepackage[T1]{fontenc}
\usepackage{url}
\usepackage{multirow}
\usepackage{xcolor}
\usepackage{tikz}
\usepackage{graphicx}
\usepackage{caption}
\usepackage{subcaption}
\usepackage{amsthm}
\usepackage{amssymb}
\usepackage{amsfonts}

\newtheorem{theorem}{Theorem}
\newtheorem{definition}{Definition}
\newtheorem{proposition}{Proposition}
\newtheorem{lemma}{Lemma}
\newtheorem{remark}{Remark}

\usetikzlibrary{positioning, fit}

\renewcommand{\le}{\leqslant}
\renewcommand{\ge}{\geqslant}
\renewcommand{\to}{\rightarrow}

\newcommand{\EE}{{\mathbb{E}}}

\def\be#1\ee{\begin{equation}#1\end{equation}}

\def \CvM{\mathrm{CvM}}

\def \Id{\mathbf{I}}
\def\be#1\ee{\begin{equation}#1\end{equation}}

\def\bqa{\begin{eqnarray}}
\def\eqa{\end{eqnarray}}

\newcommand{\bd}{\begin{displaymath}}
\newcommand{\ed}{\end{displaymath}}
\newcommand{\ba}{\begin{eqnarray}}
\newcommand{\ea}{\end{eqnarray}}

\def\bX{{\bf X}}

\def \CvM{\mathrm{CvM}}

\DeclareMathOperator{\diag}{diag}

\def \PP{\mathcal{P}}

\def \erre{\mathbb{R}}

\usepackage{hyperref}
\usepackage{booktabs}
\usepackage[inline]{enumitem}
\usepackage{verbatim}

\usepackage{caption}
\author{
  Gennaro Auricchio \\ University of Padova \\ Department of Mathematics \\ \texttt{gennaro.auricchio@unipd.it}
  \and
  Adelaide Emma Bernardelli \\ University of Salerno \\ Department of Economics and Statistics \\ IUSS Pavia \\ \texttt{adelaide.bernardelli@iusspavia.it}
  \and
  Paolo Giudici \\ University of Pavia \\  Department of Economics and Management\\\texttt{paolo.giudici@unipv.it}
  \and
  Giuseppe Toscani \\ University of Pavia \\  Department of Mathematics\\ \texttt{giuseppe.toscani@unipv.it}
}

\date{}

\begin{document}


%

\title{On Rank Graduation Metrics for High Dimensional Ordinal Data}








\maketitle


\begin{abstract}
    Evaluating the reliability of machine learning classifications  remains a fundamental challenge in Artificial Intelligence (AI), particularly when the target variable is  multidimensional.
Classification variables can be expressed by means of a categorical scale which, at best, is ordinal.
Because ordinal data lack a natural metric structure in their underlying space, most conventional distance measures aimed at assessing the accuracy of machine learning classifications cannot be directly or meaningfully applied. 
In this paper, we develop a mathematical framework for comparing ordinal data based on a family of Rank Graduation $(\mathrm{RGX}_p)$ \emph{metrics}. 
We demonstrate that these metrics can quantify the proportion of variability of the response explained by the predictions, in a similar manner as the predictive $R^2$ for continuous response variables.
After establishing theoretical connections between the $\mathrm{RGX}_p$ family and other prominent metrics in AI, we conduct extensive experiments across diverse datasets and learning tasks to evaluate their empirical performance. 
The results underscore the versatility, interpretability, and robustness of the $\mathrm{RGX}_p$ metrics as a principled foundation for developing trustworthy and SAFE AI systems.
\end{abstract}







\section{Introduction}

Making reliable predictive classifications is one of the fundamental objectives of Artificial Intelligence (AI), encompassing applications such as medical diagnosis, credit rating, educational evaluation, and customer satisfaction analysis.
All these tasks can be modelled as classification problems: given a collection of  input variables, expressed in different scales, and an output variable, expressed by a set of ordered categories (classes), a machine learning model learns the relationships between the inputs and the outputs and, consequently,  an artificial intelligence decides how to classify a given unit.

While the rapid advancement of AI has substantially improved performance in such domains, it has also raised concerns about reliability, robustness, and transparency. 
Misuse or misinterpretation of AI models can lead to unethical or even harmful outcomes,\cite{Bengio2025} which motivates the development of international frameworks and regulatory initiatives aimed at ensuring that AI is deployed in a trustworthy and responsible manner.\cite{EUAIAct2022,NIST2023AIRMF,OECD2022ClassificationAISystems}

These frameworks typically identify three core parameters for assessing the trustworthiness of an AI model: \emph{accuracy}, \emph{robustness}, and \emph{explainability}.
%
%

%
The \emph{accuracy} of a model reflects the reliability of its outputs.\cite{yin2019understanding} 
In classification tasks, accuracy is often expressed as the probability of correctly labelling a given input.
Since AI models typically produce probabilistic outputs, accuracy is commonly evaluated through discrepancy functions that measure the deviation between the model’s predictions and the correct labels.\cite{muthukumar2021classification}

The \emph{robustness} of a model characterises its resilience to external interference. 
AI systems are frequently deployed across cloud infrastructures, high-performance computing environments, and edge devices, all of which may present vulnerabilities that can be exploited to manipulate or erase training data.\cite{Aldasoro2022CyberRisk} 
For example, this happens in federated machine learning.\cite{bouacida2021vulnerabilities,yang2019federated}  
To mitigate such risks, it is crucial to develop methods for evaluating how susceptible a model’s predictions are to malicious alteration.\cite{song2025fedadmm}

The \emph{explainability} of a model is a key component of its transparency, and describes the ability to identify the main factors driving the model's predictions.
Indeed, being able to explain the output of an AI model directly addresses the issues of accountability and control of complex AI,\cite{Bengio2025,Bracke2019Explainability,Lundberg2017SHAP,Ribeiro2016WhyTrustYou,zuazua2025machine} an issue recognised as a cornerstone of trustworthy AI by major international governance frameworks.\cite{EUAIAct2022,NIST2023AIRMF,OECD2022ClassificationAISystems}
Moreover, defining explainable AI models promotes sustainability by encouraging simpler, more resource-efficient models,\cite{Babaei2022ExplainableFintech} enhances fairness,\cite{Agarwal2023BiasLending,Chen2024FairnessRatings,Mitchell2021AlgorithmicFairnessReview,Passach2022FairnessSurvey} and supports privacy preservation by analysing the effects of variable removal or noise perturbation on predictive performance.\cite{Li2020GCNPrivacy,LiuTerzi2010PrivacyScores}
Despite their conceptual clarity, these three descriptors are not always easy to evaluate within a unified mathematical framework. 
In particular, the data and representations involved in each aspect often differ in structure and meaning, making direct comparison difficult.
This challenge becomes even more pronounced when dealing with \emph{ordinal data}—data characterised by an intrinsic order but lacking a natural notion of numerical difference. 
Ordinal data are pervasive in real-world AI applications, and especially in those involving predictive classification:\cite{frank2001simple,gutierrez2015ordinal} physicians rate symptom severity, teachers assign performance levels, banks assigns ratings to borrowers, consumers express satisfaction using ordered categories, and agents express preferences over a set of items or locations.\cite{anshelevich2021ordinal,hanjoul1987facility}
The assignment of arbitrary numerical values to ordinal categories can introduce distortions that undermine both accuracy and interpretability, potentially yielding biased results and rendering the treatment of such data inherently delicate.
%
%
%
Indeed, ordinal data occupy a special position among non-Euclidean types as they encapsulate a more generic notion of preference.\cite{cheng2008neural,gutierrez2014ordinal,liu2018constrained}
In this paper, we address this issue by proposing a unified mathematical framework for comparing ordinal data for both the univariate and the multivariate case. 
The proposed approach provides a principled way to assess similarity or discrepancy between ordered observations without imposing arbitrary metric assumptions. 
This contributes not only to improving \emph{accuracy}, but also to enhancing \emph{robustness} and \emph{explainability}, as it allows one to reason transparently about the relationships between input features  and model outputs.

\subsection{Related Works}
The first dimension of AI model assessment is accuracy. 
In regression analysis, model accuracy is typically assessed using the Mean Squared Error (MSE) and its normalised counterpart, the predictive $R^2$.\cite{Gneiting2011} 
For classification tasks, evaluation commonly relies on metrics derived from the confusion matrix—such as Accuracy, F1 score, Area Under the Curve (AUC), and the Precision–Recall curve.\cite{HandTill2001MulticlassAUC} 
The Rank Graduation Accuracy Metric, namely $\mathrm{RGA}$, is a metric that builds on the concept of concordance, originally derived from the Lorenz curve and the Gini coefficient,\cite{Gini1921InequalityIncome,Lorenz1905Concentration} and extends the AUC framework to provide a single, coherent measure of predictive performance suitable for binary, ordinal, and continuous responses.\cite{GiudiciRaffinetti2024RGA}
For binary outcomes, $\mathrm{RGA}$ coincides with AUC; for ordinal and continuous responses, it serves as an analogous ranking-based measure of accuracy. 
For structured data, as images, there are several \textit{ad-hoc} introduced tools\cite{6467150} along with a plethora of divergences borrowed from kinetic theory,\cite{auricchio2025kinetic} which have found a natural home in artificial intelligence and other applied fields.\cite{bellomo2024life,bellomo2024herbert}
Aside from the standard Kullback-Leiber entropy,\cite{kullback1951information} we recall 
\begin{enumerate*}[label=(\roman*)]
    \item the Energy Distance, which has been extensively studied for statistical purposes\cite{rizzo2016energy,szekely2013energy} and recently has been rediscovered due to its appeal to study evolutionary problems\cite{auricchio2025energy} and its mathematical properties;\cite{bellemare2017cramer}
    \item the Wasserstein Distance,\cite{villani2008optimal} which has been thoroughly studied for its applicability to imaging problems,\cite{frogner2015learning,rabin2011wasserstein,tartavel2016wasserstein} to clustering,\cite{auricchio2019computing,ye2017fast,zhuang2022wasserstein} and to multiagent based systems\cite{auricchio2024extended,scagliotti2025normalizing} due to the recent developing of efficient computing methods;\cite{altschuler2019massively,auricchio2018computing,cuturi2013sinkhorn} and
    \item the Fourier Based Metrics,\cite{gabetta1995metrics} which compare the probability measures in the space of phases.\cite{auricchio2020equivalence,auricchio2023fourier,carrillo2007contractive}
\end{enumerate*} 

A second key dimension of AI model assessment concerns \emph{robustness}. 
The integration of AI into cloud infrastructures, high-performance computing environments, and edge devices introduces vulnerabilities that can be exploited to manipulate or delete training data.\cite{Aldasoro2022CyberRisk}
To explicitly account for robustness within risk management frameworks, the probability of misclassification due to a security breach can be represented as a corresponding loss in predictive accuracy (e.g., in $R^2$ or $AUC$ terms). 
%


%
The third indicator, addressing explainability, is perhaps the less conventional quantity to capture via a divergence.
A possible approach to measure how much one feature in the data affects the model performances is to compare the model trained with and without the selected feature.
In this way, we are able to express the explainability of a feature via a divergence.
This approach somewhat mimics the technique used to assess the robustness: we introduce a perturbation in the dataset (i.e., we remove a feature) and compare the performances of the model.
A different approach relies on Shapley values, a concept coming from cooperative game theory.\cite{mitchell2022sampling,rozemberczki2022shapley}
Through Shapley values, it is possible to quantify how much the feature contributes to the training of the model.
While mathematically sound, this approach is hindered by the computational cost of retrieving the Shapley values, thus such values are usually approximated via Monte Carlo methods.

The joint assessment of accuracy, robustness and explainability of classification models is difficult, and becomes particularly challenging to handle when the target response variable is multivariate. 
Part of the difficulty is due to the absence of a standard definition for the multivariate Lorenz curve; instead, several alternative formulations have been proposed,\cite{arnold2014pareto,koshevoy1996lorenz,taguchi1972two} including the concept of a Lorenz Zonoid introduced in Ref. \cite{koshevoy1997multivariate}. 
Although Lorenz Zonoids provide a mathematically and geometrically sound generalisation of the one-dimensional Gini coefficient, they are often difficult to implement, primarily due to computational complexity. 
Furthermore, they do not possess properties analogous to those of their unidimensional counterparts.\cite{decancq2012inequality} 
For example, the \emph{scaling invariance} property, which holds in the one-dimensional case, is typically not preserved by multidimensional Lorenz Zonoids.
An approach that allows to overcome the lack of the scale invariance property has been proposed in Ref. \cite{Auricchio2025WhiteningGini}, where the authors use a suitable whitening procedure to reduce the study to many uncorrelated one-dimensional distributions.
This approach has been proven to be ductile as it allows to define divergences that are scale invariant with respect to two or more arguments.\cite{Auricchio2025MultivariateGini}

\section{Preliminaries}
\label{sec2}

In this section, we introduce the main mathematical concepts we will be using throughout the paper and fix the main notation.
Throughout most of our discussion, we will deal with $N$ dimensional vectors $Y:=(Y_1,Y_2,\dots,Y_N)$ and denote with $F_n$ the empirical c.d.f. associated with the values in $\{Y_i\}_{i\in[n]}$.
By definition, the ranks of the observations, namely $r_i$, are such that $r_i=F_n(Y_i)$, thus
\[
    Y_{r_1}\le Y_{r_2}\le \dots \le Y_{r_n}
\]
are the order statistics.

\subsection{The Gini Index}

The Gini index, also known as the Gini coefficient, is a statistical measure of uneveness within a distribution, most commonly applied to quantify income or wealth inequality.\cite{Gini1914Concentrazione,Gini1921InequalityIncome}
Formally, for a population with cumulative distribution function $F(x)$ of income $x$, the Gini index $G$ can be expressed as
\begin{equation}
    \label{eq:lorentz_gini}
    G = 1 - 2 \int_{0}^{1} L(p) \, dp,    
\end{equation}
where $L(p)$ denotes the Lorenz curve, representing the cumulative share of income earned by the bottom $p$-fraction of the population, that is
\[
        L(p)=\frac{1}{m}\int_{-\infty}^{F^{[-1]}(p)}tdF(x),
\]
where $m$ is the average of $F$.
For a finite sample $Y=(Y_1, Y_2, \ldots, Y_N)$ the Lorentz curve is the unique piecewise linear function that interpolates the points
\[
    \Big(\frac{k}{N},\sum_{i=1}^kY_{r_i}\Big),
\]
where $r_i$ is the increasing reordering of $Y$, that is
\[
    Y_{r_1}\le Y_{r_2}\le Y_{r_3}\le \dots \le Y_{R_N}.
\]

Notice that the formula in \eqref{eq:lorentz_gini} can  be rewritten as
\[
G(Y) = \frac{1}{N\sum_{i=1}^NY_i}\int_{0}^1|L_Y^c(t)-L_Y(t)|dt,
\]
where $L_Y^c$ is the unique piecewise linear function that interpolates the points
\[
    \Big(\frac{k}{N},\sum_{i=N+1-k}^NY_{r_i}\Big)=\Big(\frac{k}{N},\sum_{i=1}^KY_{r_{N+1-k}}\Big),
\]
which is also known as the dual Lorentz Curve.\cite{Babaei2025SAFEAI}
It is easy to see that if $X$ and $Y$ are two $N$ dimensional vectors such that $\sum_{i=1}^NX_i =\sum_{i=1}^N Y_i$ then $L_X(t)\ge L_Y(t)$ for every $t\in[0,1]$ if and only if
\begin{equation}
\label{eq:relationLXLY}
    \sum_{i=1}^kX_{r'_i}\ge \sum_{i=1}^kY_{r_i}
\end{equation}
for every $k=1,\dots,N$, where $r_i'$ is the ordering induced by the vector $X$.
Notice that if \eqref{eq:relationLXLY} holds, then we have that $L_X^c(t)\le L_Y^c(t)$ for every $t\in[0,1]$.
The value of $G$ ranges from $0$ (perfect equality) to $1$ (maximal inequality), providing a normalized and dimensionless measure of dispersion that is widely employed in economic and statistical analyses.\cite{Gini1914Concentrazione,Gini1921InequalityIncome}

\subsection{The Rank Graduation Metrics}

Given two $N$ dimensional vectors $Y=(Y_1,Y_2,\dots,Y_N)$ and $Z=(Z_1,Z_2,\dots,Z_N)$ we define the concordance curve between $Y$ and $Z$ as the unique piecewise linear interpolation between the points
\begin{equation}
    \Big(\frac{k}{N},\sum_{i=1}^kY_{r^*_i}\Big),
\end{equation}
where $r^*_i$ is the ordering induced by the vector $Z$.
We denote the concordance curve by $L_{Y,Z}:[0,1]\to[0,\sum_{i=1}^N Y_{r^*_i}]$.

Then, a Rank Graduation metric between $Y$ and $Z$ (that we name $\mathrm{RGX}$) can be defined as the area between the Lorentz curve associated with $Y$, $L_Y$, and the concordance curve between $Y$ and $Z$,  $L_{Y,Z}$, normalized by the Gini index of $Y$.
More formally, we have that
\begin{equation}
    \mathrm{RGX}(Y,Z)=\frac{1}{G(Y)\sum_{i=1}^N Y_i}\int_{0}^1|L_Y^c(t)-L_{Y,Z}(t)|dt=1-\frac{\int_{0}^1|L_{Y,Z}(t)-L_Y(t)|dt}{\int_{0}^1|L_{Y}^c(t)-L_Y(t)|dt},
    \label{RGX}
\end{equation}
where $G(Y)$ is the Gini index of the random variable $Y$.

\subsection{The Whitening Process}
\label{sec:whitening_processes_refined}

In what follows, we denote by $\mathcal{P}_{\Id}(\erre^n)$ the subset of $\mathcal{P}(\erre^n)$ composed of probability measures with identity covariance matrix $\Id$.

\begin{definition}
    A \emph{whitening process} is a linear mapping $\mathcal{S}$ from $\bX\in\mathcal{P}(\erre^n)$ to $\bX^{*} \in \mathcal{P}_{\Id}(\erre^n)$:
    \begin{equation}
        \label{whi1_refined}
        \bX^{*} = \mathcal{S}(\bX) = W_{\bX}\bX = W_{\mu}\bX,
    \end{equation}
    where $W_\bX = W_\mu$ is an $n\times n$ matrix depending on $\bX\sim\mu$. The matrix $W_\mu$ is referred to as the \emph{whitening matrix} associated with $\mathcal{S}$.\cite{Kessy2018Whitening}
\end{definition}

If the covariance matrix $\Sigma_\mu$ of $\bX$ is invertible, the whitening matrix in~\eqref{whi1_refined} satisfies
\begin{equation}
    \label{whi2_refined}
    W_\mu^T W_\mu = \Sigma_\mu^{-1}.
\end{equation}
Condition~\eqref{whi2_refined} does not determine $W_\mu$ uniquely.
Indeed, if $W_\mu$ satisfies~\eqref{whi2_refined}, then any matrix of the form $\widetilde{W}_\mu = ZW_\mu$, with $Z$ orthogonal (i.e., $Z^TZ = \Id$), is also a valid whitening matrix. 
Consequently, several whitening schemes can be defined for a given random vector, and the specific choice depends on the desired properties of the process.\cite{LiZhang1998Sphering}

A key property of certain whitening procedures is \emph{Scale Stability}, ensuring invariance of the whitening transformation under rescaling of the coordinate axes.

\begin{definition}[Scale Stability]
    A whitening process $\mathcal{S}$ is said to be \emph{Scale Stable} if
    \begin{equation}
        \label{eq:scale_stab_refined}
        \mathcal{S}(\bX) = \mathcal{S}(Q\bX),
    \end{equation}
    for every random vector $\bX$ and for any diagonal matrix $Q=\diag(q_1,\dots,q_n)$ with strictly positive entries $q_i>0$ for all $i=1,\dots,n$.
\end{definition}

Examples of scale-stable whitening processes include the \emph{Cholesky Whitening} and the \emph{Zero-Component Analysis} of the correlation matrix (ZCA-cor) whitening.\cite{Auricchio2025MultivariateGini,Auricchio2025WhiteningGini}

\section{The Cramér-Von Mises as a Loss Function}

In this section, we consider the Cramér-von Mises divergence, highlighting its connection with other remarkable metrics and study its properties.

\begin{definition}
    Given any two cumulative distribution functions $F_X,F_Y:\mathbb{R}\to[0,1]$, we define the $p$-th order Cramér - Von Mises divergence between $F_X$ and $F_Y$ as
    \begin{equation}
        \CvM_p (F_X,F_Y) = \int_{-\infty}^{\infty} \big|F_X(u)-F_Y(u)\big|^pdF_X(u).
        \label{Cramér}
    \end{equation}
\end{definition}

%
%
Note that, with a slight abuse of notation, we use either the c.d.f. $F_X$ or its associated random variable as an input to the Cramér-Von Mises metric, so that $\CvM_p(F_X,F_Y)$ and $\CvM_p(X,Y)$ will be used interchangeably as long as $F_X$ is the c.d.f. associated with $X$ and $F_Y$ is the c.d.f. associated with $Y$.

\subsection{The Relation of the Cramér--von Mises with the Other Divergences}

We begin our study considering the Cramér-Von Mises and relating it with the Wasserstein Distance and the Energy Distance.
Given $X\sim \mu$ and $Y\sim\nu$, we define the concordance function between $X$ and $Y$ as the following map
\begin{equation}
    C_{X,Y}:[0,1]\to[0,1]\quad\quad\text{where}\quad\quad C_{X,Y}(q)=F_X(F_Y^{[-1]}(q)).
\end{equation}
In other words, the concordance function $C_{X,Y}$ is the map for which
\begin{equation}
    F_Y^{[-1]}(q)=F_X^{[-1]}(C_{X,Y}(q)),
\end{equation}
that is, $C_{X,Y}$ assigns to every $q\in[0,1]$ the same probability that $\mu$ assigns to the $q$-th percentile of the probability measure $\nu$.
%
%
When both $X$ and $Y$ are supported over $\erre$ and are absolutely continuous, we have $F_X(F^{[-1]}_Y)$ is a cumulative distribution function.
%
%
%
%
Indeed, we have that $\lim_{t\to 1^-}F_X(F^{[-1]}_{Y}(t))=1$.
Moreover, since $F_{Y}^{[-1]}$ and $F_X$ are both monotone increasing, their composition is also monotone increasing.
To conclude, we need to show that $F_X(F_Y^{[-1]})$ is continuous from the right, that is
\[
    \lim_{t\to u^+}F_X(F_Y^{[-1]}(t))=F_X(F_Y^{[-1]}(u)).
\]
This follows from the fact that $Y$ is absolutely continuous.
%

%
We denote by $\mathcal{C}_{X,Y}$ the random variable associated with the cumulative distribution function $C_{X,Y}:=F_X(F_Y^{[-1]})$.
By definition, we have that $\mathcal{C}_{X,Y}$ takes values in $[0,1]$ with probability $1$.
{
Notice that it is possible to drop the assumption on the support of both measures, however, in this case, we need to define $C_{X,Y}(t):=F_X(F_Y^{[-1]})(t)$ for every $t\in[0,1)$ and $C_{X,Y}(1):=1$.
%
}
In the next result, we relate the Cramér-Von Mises of any order with the Wasserstein Distance of order $p$ and, when $p=2$, with the Energy Distance.

\begin{proposition}
    Let $X\sim\mu$ and $Y\sim\nu$ be two absolutely continuous random variable and their associated probability distribution over $\erre$.
    %
    %
    Let us denote by $\mathcal{C}_{X,Y}$ the random variable defined as $(C_{X,Y})_{\#}\mathcal{U}$, where $\mathcal{U}$ is the uniform distribution over $[0,1]$ and $(\circ)_{\#}$ is the push-forward operator.
    We then have the following relation
    \begin{equation}
        \label{eq:CvM_and_W_p}
        \CvM_p(\mu,\nu)=W_p(\mathcal{U},\mathcal{C}_{X,Y}),
    \end{equation}
    where $\mathcal{U}$ is a uniform distribution and $\mathcal{C}_{X,Y}$ a distribution whose c.d.f is given by $F_X(F_Y^{[-1]}(t))$.
    Moreover, if $p=2$, we have that
    \begin{equation}
        \CvM_2(\mu,\nu)=\mathcal{E}(\mathcal{U},\mathcal{C}_{X,Y}).
    \end{equation}
\end{proposition}

\begin{proof}
    Let $\mu$ and $\nu$ be two absolutely continuous distributions.
    By definition, we have that
    \begin{equation}
    \label{eq:rewrite_CVM}
        \CvM_p(\mu,\nu)^p=\int_{\erre}|F_\mu(x)-F_\nu(x)|^pdF_\mu(x)=\int_{0}^1|t-F_\nu (F^{[-1]}_\mu(t))|^pdt,
    \end{equation}
    where in the last equality we have used the change of variable $t=F_\mu(x)$.
    Notice now that the function $F_\mu(F^{[-1]}_\nu(t))$ is a cumulative distribution function associated with $\mathcal{C}_{X,Y}$.
    Owing to the characterization of the solution of the one dimensional optimal transportation problem, we have that the optimal transportation map between the uniform distribution and $\mathcal{C}_{X,Y}$ is 
    \[
        T(t):= F_{\mathcal{C}_{X,Y}}^{[-1]}(F_{\mathcal{U}}(t))=(F_\mu(F^{[-1]}_\nu))^{[-1]}(t)=F_\nu(F_\mu^{[-1]}(t)).
    \]
    We thus conclude that identity \eqref{eq:CvM_and_W_p} holds.

    Lastly, we observe that if $p=2$, equation \eqref{eq:rewrite_CVM} reads as
    \begin{equation}
        \CvM_2(\mu,\nu)=\int_{0}^1|t-F^{[-1]}_\mu(F_\nu(t))|^2dt.
    \end{equation}
    Since $U(t)=t$ is the c.d.f. of the uniform distribution and $t\to F^{[-1]}_\mu(F_\nu(t))$ is the c.d.f. associated with $\mathcal{C}_{X,Y}$, we conclude the proof.
\end{proof}

\subsection{Two Tradeoff Analyses for the Cramér-Von Mises Divergence}

In this section, we present two Bias-Variance tradeoffs that hold valid for machine learning models using the Cramér-von Mises $\CvM_2$ as loss function.
The first tradeoff we propose is the classic Variance-Bias tradeoff, in which the error between the trained model and the real distribution is decomposed as the sum of a term that measures the divergence between the trained model and the average model plus the divergence between the average model and the real model.
In the second tradeoff, we consider the error induced by a not-fully trained model and propose a tradeoff that accounts for any early stops in the training process.

\subsubsection{Classic Bias-Variance Tradeoff}

Let $Y$ be a real-valued random variable with an associated (unknown) cumulative distribution function (CDF) $F$. 
Consider a dataset partition $D$ used to train a machine learning model. 
From this partition, one can define an optimal model $F_D$, obtained by minimizing the Cramér--von Mises loss over the convex set $\mathcal{F}_D$ of all CDFs that are learnt from $D$ to approximate $F$:
\[
    F_D = \arg\min_{G \in \mathcal{F}_D} \int_{\mathbb{R}^d} (G - F)^2 \, dF,
\]
where $F$ is the true underlying distribution. In practice, after training the model on $D$, we obtain a specific cumulative distribution function $F_T \in \mathcal{F}_D$ representing the learned approximation of $F$ from the dataset $D$ available during the training.
If we assume that the partition $D$ is repeatedly sampled from a random variable $\xi$, we can define an average trained model as $\hat F=\EE_\xi[F_D]$.
The average $\hat F$ allows us to decompose the expected training error of the model with respect to the Cramér-Von Mises divergence as the sum of the error of the average trained model, with respect to the true model (bias), plus the sample variance of $F$ deriving from  $\xi$.
%




\begin{theorem}
\label{thm:first_tradeoff_VB}
    We have that
    \begin{align}
        \nonumber\mathbb{E}_\xi\bigg[\int_{\erre^d}(F_D-F)^2dF\bigg] &= \mathbb{E}_\xi\bigg[\int_{\erre^d}(F_D-\hat F)^2dF\bigg] + \mathbb{E}_\xi\bigg[\int_{\erre^d}(\hat F-F)^2dF\bigg]\\
        &=\mathbb{E}_\xi\bigg[\int_{\erre^d}(F_D-\hat F)^2dF\bigg] + \int_{\erre^d}(\hat F-F)^2dF,
        \label{biasvariance}
    \end{align}
    regardless of the probability distribution $\xi\in\PP(\mathcal{D})$.
\end{theorem}

Equation \ref{biasvariance} defines a bias variance trade-off. As we increase model complexity, for example increasing the number of parameters of the trained model $F$, bias is reduced, but variance increases: it is not possible to reduce both.  

\begin{proof}
    For any $D\in\mathcal{D}$, we have
    \begin{equation}
        (F_D-F)^2=(F_D-\hat F)^2+(\hat F- F)^2+2(F_D-\hat F)(\hat F- F).
    \end{equation}
    If we take the expected value with respect to $\xi$, we have that
    \begin{equation}
    \label{eq:decomp_expected}
        \mathbb{E}_\xi\Big[(F_D-F)^2\Big] = \mathbb{E}_\xi\Big[(F_D-\hat F)^2\Big] +\mathbb{E}_\xi\Big[(\hat F-F)^2\Big]
    \end{equation}
    since
    \begin{equation}
        \mathbb{E}\Big[(F_D-\hat F)(\hat F- F)\Big] = (\hat F- F)\mathbb{E}\Big[(F_D-\hat F)\Big] = (\hat F- F)\mathbb{E}(\big[F_D\big]-\hat F) = 0
    \end{equation}
    by definition of $\hat F$.
    Finally, if we integrate \eqref{eq:decomp_expected} with respect to $dF$, we have that
    \begin{equation}
        \int_{\erre^d}\mathbb{E}_\xi\Big[(F_D-F)^2\Big]dF = \int_{\erre^d}\mathbb{E}_\xi\Big[(F_D-\hat F)^2\Big]dF +\int_{\erre^d}\mathbb{E}_\xi\Big[(\hat F-F)^2\Big]dF.
    \end{equation}
    To conclude the thesis, it suffices to swap the integral order through Tonelli Fubini's theorem.
\end{proof}

\subsubsection{A Global Bias-Variance Tradeoff}

We then come to propose a different tradeoff in which we account for the fact that the training procedure of a model is often based on numerical approximations.
This trade-off decomposes the error we commit between the best possible model that can be generated by the data available, which we denote by $F_D$, the true model $F$, and a suboptimal model $F_T$, a numerical approximation of $F_D$.
In particular, we are able to decompose the error induced by $F_T$ as the sum of the error due to the approximation and the error due to the estimation.
Before we introduce the second tradeoff, we introduce a classic technical lemma.

\begin{lemma}
\label{lmm:decomoposition_erro}
    For any dataset $D$, we have that
    \begin{equation}
        \int_{\erre^d}(H-F_D)(F_D-F)dF=0
    \end{equation}
    for every $H\in\mathcal{F}_D$.
\end{lemma}

\begin{proof}
    Let us fix a dataset partition $D$.
    First, since $\mathcal{F}_D$ is convex, we have that, given $F_D$ and given any $H\in\mathcal{F}_D$, it holds
    \[
        (1-\lambda)F_D+\lambda H = F_D+\lambda(H-F_D)\in\mathcal{F}_D
    \]
    for every $\lambda\in[0,1]$.
    Given $H\in\mathcal{F}_D$, we define $H_\lambda=F_D+\lambda(H-F_D)$.
    Since $F_D$ minimizes the function
    \begin{equation}
        G\to\int_{\erre^d}(G-F)^2dF,
    \end{equation}
    we have that the function
    \begin{align}
        \label{eq:lambda_euqation}
        V_H:\lambda &\to \int_{\erre^d}(H_\lambda-F)^2dF = \int_{\erre^d}(F_D-F+\lambda(H-F_D))^2dF\\
       \nonumber &=\int_{\erre^d}(F_D-F)^2dF+2\lambda\int_{\erre^d}(H-F_D)(F_D-F)dF+\lambda^2\int_{\erre^d}(H-F_D)^2dF
    \end{align}
    attains minimum in $\lambda =0$.
    In particular, by taking the derivative of \eqref{eq:lambda_euqation}, we must have that
    \begin{equation}
        \partial_\lambda (V_H)_{\Big|\lambda=0}=0.
    \end{equation}
    A simple computation shows that
    \begin{equation}
         \partial_\lambda (V_H)=2\bigg( \int_{\erre^d}(H-F_D)(F_D-F)dF+\lambda\int_{\erre^d}(H-F)^2dF\Bigg),
    \end{equation}
    thus it must be
    \begin{equation}
        2 \int_{\erre^d}(H-F_D)(F_D-F)dF=0,
    \end{equation}
    which concludes the proof.
\end{proof}

As a corollary of the previous proposition, we have that for any model $F_T$ obtained by training a ML model over a data partition $D$ the error induced by the model $F_T$ can be decomposed as the sum of a sub-optimality error and an error due to the estimation, based on the dataset $D$.
Noticeably, this relation holds even if we take a distribution over the set of possible partitions $D$.
\begin{theorem}
Let us denote by $F_T$ the model obtained from a machine learning system and let us denote by $F_D$ the optimal model obtainable from a dataset $D$.
Then we have that
\begin{equation}
\label{eq:decomposition_error}
    \int_{\erre^d}(F_T-F)^2dF=\int_{\erre^d}(F_T-F_D)^2dF+\int_{\erre^d}(F_D-F)^2dF.
\end{equation}
If the dataset $D$ is sampled from a random variable, namely $\xi$, we have that
\begin{equation}
\label{eq:decomposition_error_avg}
    \int_{\erre^d}(F_T-F)^2dF=\EE\Bigg[\int_{\erre^d}(F_T-F_D)^2dF\Bigg]+\EE\Big[\int_{\erre^d}(F_D-F)^2dF\Bigg].
\end{equation}
\end{theorem}

\begin{proof}
    By definition, we have that
    \begin{equation*}
        \int_{\erre^d}(F_T-F)^2dF=\int_{\erre^d}(F_T-F_D)^2dF+\int_{\erre^d}(F_D-F)^2dF+2\int_{\erre^d}(F_T-F_D)(F_D-F)dF.
    \end{equation*}
    Owing to Lemma \ref{lmm:decomoposition_erro}, we then infer \eqref{eq:decomposition_error}.
    The relation in \eqref{eq:decomposition_error_avg} follows by the linearity of the expected value.
\end{proof}

To conclude, notice that, owing to Theorem \ref{thm:first_tradeoff_VB}, the identity in \eqref{eq:decomposition_error_avg} can be rewritten as
\begin{equation*}
    \int_{\erre^d}(F_T-F)^2dF=\EE\Bigg[\int_{\erre^d}(F_T-F_D)^2dF\Bigg]+\EE\Big[\int_{\erre^d}(F_D-\hat F)^2dF\Bigg]+\int_{\erre^d}(\hat F- F)^2dF,
\end{equation*}
where $\hat F=\mathbb{E}_{\xi\sim D}[F_D]$ is the average model.

This shows that the total error caan be decomposed into: approximation error, plus bias, plus variance, thereby extending the classical bias-variance trade off.

\section{The $RGX_p$ Metrics}

As we have seen in \eqref{RGX} the Rank Graduation Metrics is the ratio between two areas
\begin{equation}
    \mathrm{RGX}(Y,Z)=1-\frac{\int_{0}^1|L_{Y,Z}(t)-L_Y(t)|dt}{\int_{0}^1|L^c_{Y}(t)-L_Y(t)|dt}=1-\frac{\frac{1}{N\bar Y}\int_{0}^1|L_{Y,Z}(t)-L_Y(t)|dt}{\frac{1}{N\bar Y}\int_{0}^1|L^c_{Y}(t)-L_Y(t)|dt},
\end{equation}
where $L_{Y,Z}(t)$ is the linear interpolation of the points in $Y$ ordered according to the ordering of $Z$ and $\bar Y$ is the mean value of $Y$, so that $N\bar Y=\sum_{i=1}^N Y_i$.
Owing to the fact that $\mathrm{RGX}(Y,Z)$ is the ratio of two positive quantities and owing to the fact that it holds
\begin{equation}
\label{eq:inequalities_linear_Lorentz}
    L_Y(t)\le L_{Y,Z}(t)\le L^c_Y(t)
\end{equation}
for every $t\in[0,1]$, we have $0\le \mathrm{RGX}(Y,Z)\le 1$.
In particular, we have that the RGX can be interpreted as the percentage of variability that can not be explained by the ordering induced by $Z$.
We can then generalize the argument to define an $l_p$ version of the $\mathrm{RGX}$ metric.

\begin{definition}
\label{def:RGX_p}
    Let $Y$ and $Z$ be two $N$ dimensional and positive vectors.
    We define the $\mathrm{RGX}_p$ metric as follows
    \begin{equation}
    \label{eq:RGX_p}
        \mathrm{RGX}_p(Y,Z)=1-\frac{\int_{0}^1|L_{Y,Z}(t)-L_Y(t)|^pdt}{\int_{0}^1|L^c_{Y}(t)-L_Y(t)|^pdt}=1-\frac{\frac{1}{(N\bar Y)^p}\int_{0}^1|L_{Y,Z}(t)-L_Y(t)|^pdt}{\frac{1}{(N\bar Y)^p}\int_{0}^1|L^c_{Y}(t)-L_Y(t)|^pdt},
    \end{equation}
    where $r^*_i$ is the ordering induced by $Z$ and $r_i$ is the ordering induced by $Y$.
\end{definition}
First, we observe that if $p=1$, the $\mathrm{RGX}_p$ coincides with the $\mathrm{RGX}$ metric introduced in Ref. \cite{Babaei2025SAFEAI}.
Moreover, whenever $p>0$, we have that \eqref{eq:inequalities_linear_Lorentz} implies
\begin{equation}
0\le (L_{Y,Z}(t)-L_Y(t))^p \le (L^c_{Y}(t)-L_Y(t))^p,
\end{equation}
so that $0\le \mathrm{RGX}_p(Y,Z)\le 1$ for every $Y$ and $Z$.
In particular, we show that $\mathrm{RGX}_p$ is the percentage of variability that is explained by the ordering of $Z$.
To do that, we first show that the $p$-th root of the denominator of \eqref{eq:RGX_p} is a variability measure.
For the sake of notation, we denote by $S_p(Y)$ the denominator of \eqref{eq:RGX_p}, that is
\begin{equation}
    S_p(Y)=\frac{1}{N\bar Y}\Big(\int_{0}^1|L^c_{Y}(t)-L_Y(t)|^pdt\Big)^{\frac{1}{p}}.
\end{equation}

\subsection{$S_p$ as a Variability Measure}

To check that $S_p$ is a variability measure, we prove that $S_p$ possesses all the six properties outlined in Ref. \cite{Hurley2009SparsityMeasures}.
This result is interesting on its own, as it defines a new variability index that, just like the Gini Index\cite{Gini1914Concentrazione,Gini1921InequalityIncome} and the $l_{q,p}$-index, possesses all the six axiomatic properties that define a variability measure.

\subsubsection{The Scale Invariance Property}
First, we check that $S_p$ is scale invariant, meaning that multiplying the entries of $Y$ by a positive constant $\gamma$ does not change the variability of the vector, that is $S_p(\gamma Y)=S_p(Y)$ for any $\gamma>0$ and any positive vector $Y$.

\begin{proposition}
    The Index $S_p$ is scale invariant.
\end{proposition}

\begin{proof}
    This follows trivially by the definition of $S_p$.
    Indeed, let $Y$ be a positive $N$-dimensional vector and $\gamma>0$, then
    \begin{align*}
        S_p(\gamma Y)&=\frac{1}{\gamma N\bar Y}\bigg(\int_{0}^1|L^c_{\gamma Y}(t)-L_{\gamma Y}(t)|^pdt\bigg)^{\frac{1}{p}}\\
        &=\frac{1}{N\bar Y}\bigg(\int_{0}^1|L^c_{ Y}(t)-L_{ Y}(t)|^pdt\bigg)^{\frac{1}{p}}=S_p(Y),
    \end{align*}
    where we have used the fact that \begin{enumerate*}[label=(\roman*)]
        \item the ordering induced by $Y$ and $\gamma Y$ is the same and
        \item $\frac{1}{\gamma N\bar Y}L_{\gamma Y}(t)=\frac{1}{N\bar Y}\gamma L_Y(t)$ for every $t\in[0,1]$ since both functions interpolate the same points.
    \end{enumerate*} 
\end{proof}

\subsubsection{The Rising Tide Property}
We then check the rising tide property, which states that if we increase all the entries of a positive vector by a common positive value $c$, the variability decreases, that is $S_p(Y+c)\le S_p(Y)$ for every $c\ge 0$.

\begin{theorem}
    The index $S_p$ has the rising tide property.
\end{theorem}

\begin{proof}
    Given a positive vector $Y$, let $c\ge 0$ be a positive constant, and let $Y_c$ be the vector defined as
    \[
        (Y_c)_i=Y_i+c
    \]
    for every $i=1,\dots,N$.
    Owing to the scale invariance of $S_p$, we can assume without loss of generality that $\sum_{i=1}^NY_i=1$.
    First, we notice that the ordering induced by $Y$ is the same one induced by $Y_c$.
    To conclude that $S_p$ possesses the raising tide property it the suffices to prove that
    \begin{equation}
    \label{eq:inequality}
        \sum_{i=1}^kY_{r_i}\le \frac{kc+\sum_{i=1}^kY_{r_i}}{1+Nc},
    \end{equation}
    as this implies $\frac{L_{Y}(t)}{\sum_{i=1}^N Y_i} \le \frac{L_{Y_c}(t)}{NC+\sum_{i=1}^N Y_i}$ for every $t\in[0,1]$, as $L_Y$ and $L_{Y_c}$ are the piecewise-linear interpolations of the points $(k/N,\sum_{i=1}^kY_{r_i})$ and $(k/N,\frac{kc+\sum_{i=1}^kY_{r_i}}{1+Nc})$.
    We have that \eqref{eq:inequality} is equivalent to
    \begin{align*}
        \sum_{i=1}^kY_{r_i}\Big(\frac{Nc}{1+Nc}\Big)=N\sum_{i=1}^kY_{r_i}\Big(\frac{c}{1+Nc}\Big)\le \frac{kc}{1+Nc}=k\frac{c}{1+Nc},
    \end{align*}
    or equivalently
    \[
        N\sum_{i=1}^kY_{r_i}\le k,
    \]
    which holds because $r_i$ is the increasing reordering of $Y$.
    A similar argument allows us to conclude that
    \[
        L^c_{Y_c}(t) \le  L^c_{Y}(t),
    \]
    which implies $0\le (L^c_{Y_c}(t) -  L_{Y_c}(t))^p \le (L^c_{Y}(t) \le  L_{Y}(t))^p$ for every $t\in[0,1]$ and every $p>0$ and thus the thesis.
\end{proof}

\subsubsection{The Cloning Property}
We then check the Cloning property, which states that the variability measured by the index does not change if we measure the variability of $Y=(Y_1,Y_2,\dots, Y_N)$ or the variability of $Y'=(Y_1,Y_1,Y_2,Y_2,\dots,Y_N,Y_N)$.
It is easy to see that the reordering of the vector $Y'$ is given by 
\begin{equation}
    r'_i=\begin{cases}
        r_{\frac{i}{2}} \quad\quad\quad\quad \text{if $i$ is even} \\
        r_{\frac{i+1}{2}} \;\quad\quad\quad \text{otherwise}
    \end{cases}
\end{equation}
where $r_i$ is the ordering of $Y$ for every $i=1,2,\dots, N$.
We then have the following.

\begin{proposition}
    The $S_p$ index has the cloning property, that is $S_p(Y)=S_p(Y')=S_p(Y,Y)$ for any positive vector $Y$.
\end{proposition}

\begin{proof}
    Owing to the scale invariance of $S_p$, we can assume without loss of generality that $Y$ is such that $N\bar Y=\sum_{i=1}^NY_i=1$ and $Y'=(Y/2,Y/2)$ so that $2N\bar Y'=1$.
    By definition, we have that $L_{Y'}$ is the linear interpolation of the points
    \begin{equation}
    \label{eq:points_to_interpolate}
        \Big(\frac{k}{2N},\sum_{i=1}^k\frac{Y_{r_i'}'}{2}\Big),
    \end{equation}
    for every $k=1,2,\dots,2N$.
    First, we notice that if $k$ is even, so that $k=2\ell$, it holds
    \[
        L_Y(\ell/N)=\sum_{i=1}^\ell Y_{r_i}=\sum_{i=1}^{2\ell}\frac{Y_{r_i'}'}{2}=L_{Y'}(\ell/N).
    \]
    To conclude, it suffices to show that $L_{Y}\Big(\frac{\ell}{N}+\frac{1}{2N}\Big)=L_{Y'}\Big(\frac{2\ell+1}{N}\Big)$, we would then be able to conclude that $L_Y(t)=L_{Y'}(t)$ as they are both piecewise linear interpolations of the set of points in \eqref{eq:points_to_interpolate}.
    Owing to the linearity of $L_Y$, we have that
    \begin{equation*}
        L_{Y}\Big(\frac{\ell}{N}+\frac{1}{2N}\Big)=\sum_{i=1}^\ell Y_{r_i}+\frac{Y_{r_{i+1}}}{2}=\sum_{i=1}^{2\ell} \frac{Y_{r_i'}'}{2} +\frac{Y_{r_{2\ell+1}'}'}{2}=\sum_{i=1}^{2\ell+1} \frac{Y_{r_i'}'}{2}=L_{Y'}\Big(\frac{2\ell+1}{N}\Big),
    \end{equation*}
    we thus conclude the thesis.
\end{proof}

\subsubsection{The Robin Hood Property}

Lastly, we check that the index $S_p$ possesses the Robin Hood property, which states that the variability decreases when we transfer some of the quantity represented by $Y$ from a large entry to a small entry.
More formally, the Robin Hood property states that given a positive vector $Y$ and two indexes $i$ and $j$ such that
\[
    Y_i<Y_j
\]
there exists an $\bar \varepsilon>0$ such that if we increase $Y_i$ by $0\le \varepsilon\le \bar \varepsilon$ and decrease $Y_j$ by the same quantity $\varepsilon$, the total variability measured by $S_p$ decreases.
In other words, if we define
\begin{equation}
\label{eq:Y_epsi}
    Y_\varepsilon=\begin{cases}
        (Y_\varepsilon)_i=Y_i+\varepsilon\\
        (Y_\varepsilon)_j=Y_j-\varepsilon\\
        (Y_\varepsilon)_k = Y_k\quad\quad\text{otherwise}\\
    \end{cases}
\end{equation}
we have $S_p(Y_\varepsilon)\le S_p(Y)$ for $\varepsilon$ small enough.

\begin{proposition}
    The index $S_p$ has the Robin Hood property.
\end{proposition}

\begin{proof}
    Given a positive vector $Y$, let us consider $i$ and $j$ for which it holds $Y_i<Y_j$.
    Without loss of generality, assume that $Y$ is such that $Y_j\neq Y_i$ as long as $i\neq j$.\footnote{The case in which $Y_j=Y_i$ for $i\neq j$ follows by the same argument and by choosing a suitable way to break ties while determining the ordering of the values in $Y$.}
    Then, we define $\bar \varepsilon := \min_{Y_l>Y_i}Y_l-Y_i$.
    By definition of $\bar \varepsilon$, we have that the ordering of $Y_\varepsilon$ and $Y$ have the same ordering, which we denote by $r_i$.
    Moreover, we notice that for every $\varepsilon$ it holds $\bar Y_\varepsilon=\bar Y$.
    To conclude, it suffices to show that $L_Y(t)\le L_{Y_\varepsilon}(t)$ for every $t\in[0,1]$ or, equivalently, that 
    \begin{equation}
    \label{eq:sums_RH}
        \sum_{i=1}^kY_{r_i}\le \sum_{i=1}^k (Y_{\epsilon})_{r_i}
    \end{equation}
    for every $k=1,\dots,N$.
    Let us now introduce the indexes $i'$ and $j'$ which are defined as the indexes for which it holds $r_{i'}=i$ and $r_{j'}=j$.
    and
    Since $Y_i<Y_j$, we have that $i'<j'$.
    Then, we observe that \eqref{eq:sums_RH} holds as an equality whenever $k<i'$ and $k\ge j'$, while for $i'\le k<j'$ we have
    \begin{equation}
        \sum_{i=1}^k (Y_{\epsilon})_{r_i}=\varepsilon+\sum_{i=1}^kY_{r_i}\ge \sum_{i=1}^kY_{r_i}, 
    \end{equation}
    hence the thesis.
\end{proof}

\subsubsection{The Bill Gates and the Babies Properties}
Finally, we consider the Bill Gates and the Babies properties.
Roughly speaking, the Bill Gates property states that if we increase one of the entries of $Y$ indefinetively, the variability will increase.
This property describes the fact that if one person accrues a lot of a quantity, the final distribution of that quantity will be more uneven.
The Babies properties describes a similar effect described by the Bill Gates property, as it states that if we add null entries to a vector, then its variability increases.
Since we have already proven that $S_p$ possesses the scale invariancy, the rising tide, the cloning, and the Robin Hood properties, these two properties are necessarily possessed by $S_p$ (see Ref. \cite{Hurley2009SparsityMeasures}).

\begin{remark}
    Lastly, we observe that if we take the limit of $S_p$ for $p$ that goes to $+\infty$, we have
    \begin{equation}
        \label{eq:S_infty}
        S_{\infty}(Y):=\lim_{p\to\infty}S_p(Y)=\frac{1}{N\bar Y}\sup_{t\in[0,1]}|L_Y^c(t)-L_Y^c(t)|.
    \end{equation}
    By the definition of $L_Y$ and $L_Y^c$, we have that the supremum in \eqref{eq:S_infty} has an explicit form.
    Indeed, it is easy to see that the function $k\to \sum_{i=1}^k(Y_{r_{N+1-i}}-Y_{r_i})$ is increasing when $k\le \frac{N}{2}$ and decreasing otherwise; we have that 
    \[
        \sup_{t\in[0,1]}|L_Y^c(t)-L_Y^c(t)| = |L_Y^c(1/2)-L_Y^c(1/2)|,
    \]
    so that, when $N$ is even,
    \[
        S_\infty(Y)=\frac{1}{N\bar Y}\sup_{t\in[0,1]}|L_Y^c(t)-L_Y^c(t)|=\frac{\sum_{i=1}^{N/2}Y_{r_{N+1-r_i}}-\sum_{i=1}^{N/2}Y_{r_i}}{N\bar Y}.
    \]
    To conclude, we notice that $S_\infty$ is closely related with the Pietra index\cite{Pietra1915RelazioniIndici} which is defined as
    \[
        P(Y):=\sup_{t\in [0,1]}|t-L_Y(t)|.
    \]
    Indeed, owing to the fact that $L_Y(t)\le t \le L_Y^c(t)$ for every $t\in[0,1]$, it is easy to see that $P(Y)\le S_\infty (Y) \le 2 P(Y)$.
\end{remark}

\subsection{The Properties of the $\mathrm{RGX}_p$ metric} 
Now that we have established that every $\mathrm{RGX}_p$ can be interpreted as the percentage of variability induced by the prediction $Z$ over the real value $Y$, we study the properties of any $\mathrm{RGX}_p$ metric.
First, as we have already noticed, we have that $\mathrm{RGX}_p(Y,Z)\in[0,1]$.
Moreover, we have that $\mathrm{RGX}_p(Y,Z)=1$ if and only if the ordering induced by $Y$ and the ordering induced by $Z$ are the same.
Likewise, $\mathrm{RGX}_p(Y,Z)=0$ if and only if the ordering induced by $-Y$ and the ordering induced by $Z$ are the same.
Unlike the Cramér-Von Mises, the $\mathrm{RGX}_p(Y,Z)=0$ is invariant under any change of positive scale.

\begin{proposition}
    Given any couple of positive vectors $Y$ and $Z$, we have that $\mathrm{RGX}_p(Y,Z)=\mathrm{RGX}_p(\gamma Y,\gamma Z)$ for any $\gamma>0$.
\end{proposition}

\begin{proof}
    First, we notice that both $\gamma Y$ and $\gamma Z$ induce the same ordering as $Y$ and $Z$.
    Second, we have that
    \begin{equation}
        \label{eq:RGX_scale_invariance}
        \mathrm{RGX}_p(\gamma Y,\gamma Z)=1-\frac{\frac{1}{(\gamma N\bar Y)^p}\int_{0}^1|L_{\gamma Y,\gamma Z}(t)-L_{\gamma Y}(t)|^pdt}{\frac{1}{(\gamma N\bar Y)^p}\int_{0}^1|L_{\gamma Y}^c(t)-L_{\gamma Y}(t)|^pdt}.
    \end{equation}
    Since the denominator of \eqref{eq:RGX_scale_invariance} is scale invariant, it suffices to show that the numerator of \eqref{eq:RGX_scale_invariance} is scale invariant.
    To conclude, we then notice that $|L_{\gamma Y,\gamma Z}(t)-L_{\gamma Y}(t)|^p=\gamma^p |L_{Y,Z}(t)-L_{Y}(t)|^p$, which concludes the thesis.
\end{proof}

Moreover, since the value computed by $\mathrm{RGX}_p(Y,Z)$ depends on the vector $Z$ only by the ordering induced by the vector $Z$, we have that any monotone increasing transformation of the vector $Z$ does not affect the value of the metric $\mathrm{RGX}_p$.

\begin{proposition}
    Let $f:\erre\to\erre$ be a monotone increasing function.
    We then have that
    \begin{equation}
        RGX(Y,Z)=RGX(Y,f(Z)),
    \end{equation}
    where $f(Z):=(f(Z_1),f(Z_2),\dots,f(Z_N))$.
    If $f$ is monotone decreasing, it holds
    \begin{equation}
        RGX(Y,f(Z))=1-RGX(Y,Z).
    \end{equation}
\end{proposition}

\begin{proof}
    Let us consider $Z$ a vector and $f$ a monotone increasing function.
    By definition, we have that the ordering of $Y^*$ is the same one as per $f(Z)=(f(Z_1),f(Z_2),\dots,f(Z_N))$.
    Since $Z$ affects the definition of $RGX$ only by determining the reordering $r^*_i$, we conclude that $RGX(Y,Z)=RGX(Y,f(Z))$.
    Let us now consider a monotone decreasing function $f$.
    We then have that the ordering of $f(Z)$, namely $s^*_i$ is the reverse ordering of $r^*_i$, that is $s^*_i=r^*_{N+1-i}$.
    Lastly, notice that
    \begin{align*}
        \sum_{k=1}^N\bigg(\sum_{i=1}^k (Y_{s^*_i}-Y_{r_i})\bigg) &= \sum_{k=1}^N\bigg(\sum_{i=1}^k (Y_{r^*_{N+1-i}}-Y_{r_i})\bigg)-\sum_{k=1}^N\bigg(\sum_{i=1}^k (Y_{r^*_{N+1-i}}-Y_{s^*_i})\bigg)\\
        &=\sum_{k=1}^N\bigg(\sum_{i=1}^k (Y_{r^*_{N+1-i}}-Y_{r_i})\bigg)-\sum_{k=1}^N\bigg(\sum_{i=1}^k (Y_{r^*_{i}}-Y_{r_i})\bigg),
    \end{align*}
    which allows us to conclude that
    \begin{equation*}
        RGX(Y,f(Z))=1-RGX(Y,Z).
    \end{equation*}
\end{proof}

It is then easy to see that the RGX metric, intended as a function of $Y$ and $Z$, has different properties than the Cramér-Von Mises.
Regardless, in the next section, we show that it is possible to frame any $\mathrm{RGX}_p$ metric as a $p$ Cramér-Von Mises divergence by associating to the couple of the vectors $Y$ and $Z$ a suitable couple of elements of $\PP(\erre)$.

\begin{remark}
    To conclude, we show that, depending on the input of the $\mathrm{RGX}_p$ metric it is possible to capture different properties of an AI model:

\begin{itemize}
    \item \textbf{Accuracy ($\mathrm{RGA}$):} Set $Y$ as the ground truth and $Z=\hat{Y}$ as the model prediction. 
    Then we denote the $\mathrm{RGX}$ by $\mathrm{RGA}$, where $A$ stands for accuracy, and use it to measure the overall accuracy of the AI model, $\hat{Y}$.
    \item \textbf{Robustness ($\mathrm{RGR}$):} Set $Y=\hat{Y}$ as the prediction on the original data and $Z=\hat{Y}^{(p)}$ as the prediction on perturbed data. 
    Then we denote the $\mathrm{RGX}$ by $\mathrm{RGR}$, where $A$ stands for robustness, and use it to measure the overall robustness of the AI model, $\hat{Y}$.
    \item \textbf{Explainability ($\mathrm{RGE}$):} Set $Y=\hat{Y}$ as the prediction on the full dataset and $Z=\hat{Y}^{(-j)}$ as the prediction obtained after removing the $j$-th feature(s). 
    Then we denote the $\mathrm{RGX}$ by $\mathrm{RGE}$, where $E$ stands for explainability, and use it to assess the explainability of the $j$-th variable(s), with respect to the model.
    $\hat{Y}$.
\end{itemize}

\end{remark}

\subsection{Re-phrasing the $\mathrm{RGX}_p$ as a $p$ Cramér-Von Mises divergence}

Let us now consider two vectors $Y=(Y_1,\dots,Y_N)$ and $Z=(Z_1,\dots,Z_N)$.
In what follows, we denote by $\Delta_{\frac{i}{N}}$ absolutely continuous probability measure whose density is defined as
\begin{equation}
    \label{eq:Delta_i_N}
    {\Delta_{\frac{i}{N}}}(x):=N\mathbb{I}_{[\frac{i-1}{N},\frac{i}{N}]}(x),
\end{equation}
where $\mathbb{I}_A$ is the indicator function of the set $A$.
Following our notation, let us denote by $r_i$ the real ordering of $Y$ and $r^*_i$ the real ordering of $Z$.
We then introduce the following two probability measures
\begin{equation}
\label{eq:def_mu_mustar}
    \mu=\sum_{i=1}^N\frac{Y_{r_i}}{N\bar Y}\Delta_{\frac{i}{N}}\quad\quad\quad\quad\text{and}\quad\quad\quad\quad\mu^*=\sum_{i=1}^N\frac{Y_{r^*_i}}{N\bar Y}\Delta_{\frac{i}{N}},
\end{equation}
where $\bar Y$ is the average of the entries of $Y$, so that $N\bar Y=\sum_{i=1}^NY_i$.
It is immediate to check that the cumulative distribution function of $\mu$ is $L_Y$, while the cumulative distribution function associated with $\mu^*$ is $L_{Y,Z}$.
We can then express the numerator in the definition of the $\mathrm{RGX}_p$ metric as the $L_p$ norm of the difference between the cumulative distribution functions induced by $\mu^*$ and $\mu$.
In particular,
\begin{equation}
\label{eq:RGX_tosemiCVM}
    \mathrm{RGX}_p(Y,Z)=1-\frac{\frac{1}{(N\bar Y)^p}\int_{0}^1|L_{Y,Z}(t)-L_{\gamma Y}(t)|^pdt}{S^p_p(Y)},
\end{equation}
where we have used the fact that $L_{\gamma Y,\gamma Z}(t)-L_{\gamma Y}(t)\ge 0$ for every $t\in[0,1]$.
Notice that the right hand side of Equation \eqref{eq:RGX_tosemiCVM} resembles a discrete version of the Cramér-von Mises divergence, with the only difference being that measure used to integrate the difference between the two cumulative distribution functions is determined by the uniform distribution rather than with respect to $\mu$.
We therefore introduce the following weighted $\mathrm{RGX}$ metric
\begin{equation}
    \mathrm{WRGX}_{p}(Y,Z)=1-\frac{\int_{0}^1|L_{Y,Z}(t)-L_{Y}(t)|^pd\mu(t)}{\int_{0}^1|L^c_{Y}(t)-L_{Y}(t)|^pd\mu(t)}.
\end{equation}

The modified $\mathrm{WRGX}_{p}$ metric is therefore a weighted $\mathrm{RGX}_p$ metric in which the $L_p$ norm of each segment $[\frac{i-1}{N},\frac{i}{N}]$ is weighted by $\frac{Y_{r_i}}{\sum_{i=1}Y_i}$ rather than by a uniform value $\frac{1}{N}$.
By using the weights induced by the vector $Y$ the metric gives more importance to entries associated with larger values of $Y$.

By recalling the definition of the $p$-th order Cramér-von Mises divergence, we infer the following result.

\begin{proposition}
    Given any two vectors $Y$ and $Y^*$, we have that
    \begin{equation}
        \mathrm{WRGX}_{p}(Y,Z)=1-\frac{\CvM_p(\mu,\mu^*)}{\int_{0}^1|L^c_{Y}(t)-L_{Y}(t)|^pd\mu(t)}.
    \end{equation}
\end{proposition}

\begin{proof}
    It follows from the identity 
    \[
        \CvM_p(\mu,\mu^*)=\int_{\erre}|F_\mu(t)-F_{\mu^*}(t)|^pdF_\mu(t)=\int_{0}^1|L_{Y,Z}(t)-L_Y(t)|^pd\mu(t).
    \]
\end{proof}

\subsection{Multivariate Extension}

To conclude, we propose a multivariate extension of the $\mathrm{RGX}$ metric.
Unfortunately, there is no standard way to extend the Lorentz Curve to the multivariate case.
To overcome this issue, we take the same approach that has been recently proposed to extend the Gini Index to the multivariate case.\cite{Auricchio2025WhiteningGini}
Specifically, let $Y^*$ denote a whitened random vector. The multivariate Gini Index $G_{\mathrm{multi}}$ is defined as  
\begin{equation}
    G_{\mathrm{multi}}(Y)=\sum_i \lambda_i\, G(Y^*_i),
    \label{multi}
\end{equation}
where $G$ denotes the univariate Gini Index, $Y^*_i$ is the $i$-th coordinate of $Y^*$, and $\lambda_i>0$ are defined as
\begin{equation}
\label{def:lambdas}
    \lambda_i = \frac{|m_i^*|}{\sum_j |m_j^*|},
\end{equation}
where $m_i^*$ is the mean of the $i$-th coordinate of $Y^*$.
Notice that $\sum_i \lambda_i = 1$, so that $G_{\mathrm{multi}}$ is a convex combination of the one dimensional Gini Indexes of the components of the whitened vector.
We then consider the following multivariate extension of the $\mathrm{RGX}$ coefficient.

\begin{definition}
    Let $Y$ and $Z$ be two random vectors, and let $W$ be a whitening matrix for $Y$ such that $Y^* = W Y$ has uncorrelated components.  
    We define the multivariate $\mathrm{RGX}$ metric as
    \begin{equation}
        \mathrm{RGX}_p(Y,Z)=\sum_{i=1}^d \lambda_i \mathrm{RGX}_p(Y^*_i,Z^*_i),
    \end{equation}
    where:
    \begin{enumerate*}[label=(\roman*)]
        \item $Y^*$ is the whitened version of $Y$, and $Z^*$ is obtained by applying the same whitening matrix $W$ to $Z$ and
        \item $\lambda_i$ are defined as in \eqref{def:lambdas}.
    \end{enumerate*}
\end{definition}

First, note that $0\le \mathrm{RGX}_p(Y,Z) \le 1$.  
Moreover, since we are considering a scale stable whitening, $\mathrm{RGX}_p$ metric retains the scale invariancy property.

\begin{theorem}
    The metrics $\mathrm{RGA}$, $\mathrm{RGR}$, and $\mathrm{RGE}$ are scale-invariant.
\end{theorem}

\begin{proof}
    The scale invariancy property follows from the fact that we are using a scale stable whitening process, as shown in Ref. \cite{Auricchio2025WhiteningGini}.
\end{proof}




\section{Experimental Results}

We applied SAFE metrics to evaluate the predictive performance of two alternative models - a linear regression (LM) and a feedforward neural network (NN) - across both univariate and multivariate ESG specifications. 
The goal of this analysis was to assess model \emph{accuracy}, \emph{robustness}, and \emph{explainability} in predicting the Environmental (E), Social (S), and Governance (G) scores of Small and Medium-sized Enterprises (SMEs). All experiments were conducted in \textsf{R} using a five-fold cross-validation procedure to ensure robustness and generalizability.

Each model was trained on a dataset of 1,062 SMEs, where the three ESG pillars — continuous scores between 0 (worst) and 1 (best) — are used as dependent variables. 
Predictors include financial indicators and a categorical sector variable representing five industrial groups. Each financial predictor $x$ is standardized using the z-score formula:

\begin{equation*}
x^{\text{std}}_i = \frac{x_i - \bar{x}}{\sigma_x}
\end{equation*}

where $\bar{x}$ and $\sigma_x$ are the sample mean and standard deviation of $x$, respectively.
The codes for variables are described in Table \ref{tab:variables}.

\begin{table}[ht]
\centering
\caption{Description of the variables used in the experiments.}
\label{tab:variables}
\begin{tabular}{ll}
\toprule
\textbf{Code} & \textbf{Description} \\
\midrule
\texttt{E.Sc}   & Environmental Pillar Score \\
\texttt{S.Sc}   & Social Pillar Score \\
\texttt{G.Sc}   & Governance Pillar Score \\
\texttt{sTASS}          & Standardized Total Assets \\
\texttt{sSFND}          & Standardized Shareholders' Funds \\
\texttt{sTOVR}          & Standardized Turnover \\
\texttt{sNINC}          & Standardized Net Income \\
\texttt{sEBTD}          & Standardized EBITDA \\
\texttt{Sector\_factor}/\texttt{Sector} & Sector Classification \\
\bottomrule
\end{tabular}
\end{table}

\subsection{Model Estimation}

Let $y_i^{(p)}$ denote the standardized ESG score of firm $i$ for pillar $p \in \{E, S, G\}$, and let $x_i = (x_{i1}, \ldots, x_{ik})$ represent the vector of standardized predictors. For each pillar, two models were estimated - the first one is the linear regression, and the second one is the neural network:

\begin{align}
\hat{y}_i^{(p)} &= \beta_0^{(p)} + x_i' \beta^{(p)} + \varepsilon_i^{(p)}, \\
\hat{y}_i^{(p)} &= f_{\theta}^{(p)}(x_i) = W_2^{(p)} \, g(W_1^{(p)} x_i + b_1^{(p)}) + b_2^{(p)} + \varepsilon_i^{(p)},
\end{align}

where $g(\cdot)$ is the Rectified Linear Unit (ReLU) activation function and $\theta = \{W_1, W_2, b_1, b_2\}$ are the trainable parameters of the neural network optimized via backpropagation. 
The network contains a single hidden layer with five neurons in the univariate specification and ten neurons in the multivariate one, trained for up to 1,000 iterations with a linear output activation, ensuring convergence while maintaining interpretability and computational stability.
The neural network architecture was intentionally kept simple to ensure comparability with the linear model and to avoid overfitting, given the high correlation between financial variables.
In the univariate specification, a single hidden layer with five neurons was used, corresponding approximately to the number of input predictors, while the multivariate specification employed ten neurons to account for the joint modelling of three ESG dimensions.

\subsection{Univariate Specification}

In the univariate case, each ESG pillar was modeled independently using both LM and NN models. 
At each cross-validation fold, the training set was used to fit the model, and predictions were generated on the hold-out set. 
The SAFE metrics were computed using the Cramér-Von Mises distance.

RGA measures the degree of rank concordance between actual and predicted values, while RGR assesses model stability by comparing predictions before and after a small random perturbation:
\begin{equation}\label{eq:perturbation}
\hat{y}_i' = \hat{y}_i + \varepsilon_i', \qquad \varepsilon_i' \sim \mathcal{N}(0, (0.5 \, \sigma_{\hat{y}})^2),
\end{equation}
where $\sigma_{\hat{y}}$ is the standard deviation of the predicted scores and $0.5$ is a scaling factor controlling the perturbation magnitude. 
RGE quantifies the relative contribution of each predictor by comparing full-model predictions with those obtained from a reduced model excluding a single variable.

Fold-level SAFE values were collected for each model and pillar, and final estimates were obtained by averaging across folds.

\subsection{Multivariate Specification}

In the multivariate case, the three ESG pillars were modeled jointly to account for potential interdependence among them. 
Let $Y = [E.Sc, S.Sc, G.Sc]$ denote the matrix of standardized ESG scores. 
Following Auricchio et al. (2025),\cite{Auricchio2025WhiteningGini} a whitening transformation was applied to remove cross-correlation among pillars:
\begin{equation}
C = \mathrm{corr}(Y), \quad C = O D O', \quad W = O D^{-1/2} O', \quad Y^W = Y W,
\end{equation}
where $O$ and $D$ are the eigenvectors and eigenvalues of $C$, respectively. 
The resulting components $Y^W = (W_1, W_2, W_3)$ are uncorrelated with unit variance, allowing consistent multivariate estimation of SAFE metrics.

Each whitened dimension $W_j$ was assigned a weight proportional to its average absolute mean:
\begin{equation}
\lambda_j = \frac{|\bar{W}_j|}{\sum_{j=1}^{3} |\bar{W}_j|}, \qquad \sum_{j=1}^{3} \lambda_j = 1,
\end{equation}
yielding $\lambda_E = 0.4129$, $\lambda_S = 0.2721$, and $\lambda_G = 0.3150$. These weights are also displayed in Figure~\ref{fig:weights}.

The multivariate SAFE metrics were then computed as weighted averages across whitened dimensions.

\begin{align}
\text{RGA}^W &= \sum_{j=1}^3 \lambda_j \, \text{RGA}(W_j, \hat{W}_j), \\
\text{RGR}^W &= \sum_{j=1}^3 \lambda_j \, \text{RGR}(\hat{W}_j, \hat{W}'_j), \\
\text{RGE}^W_k &= \sum_{j=1}^3 \lambda_j \, \left(1 - \text{RGE}^{*}(\hat{W}_j, \hat{W}_{j,(-k)})\right).
\end{align}
where $W_j$ represents the true (observed) whitened dimension $j$ for the test set; $\hat{W}_j$ represents the predicted whitened dimension $j$ obtained from the full model; $\hat{W}'_j$ represents the perturbed prediction for whitened dimension $j$, obtained by adding noise to $\hat{W}_j$ according to equation \ref{eq:perturbation}; $\hat{W}_{j,(-k)}$ represents the predicted whitened dimension $j$ obtained from a reduced model that excludes predictor $k$.

At each cross-validation fold, both models were fitted on the training data and evaluated on the held-out subset. 
The RGA, RGR, and RGE metrics were computed fold-wise and then aggregated to obtain final mean and standard deviation values for each model and variable. 
This setup enables a direct comparison of the accuracy, robustness, and explainability of linear and nonlinear approaches, both at the individual pillar level and in the joint ESG framework.

\begin{figure}
    \centering
    \includegraphics[width=0.7\linewidth]{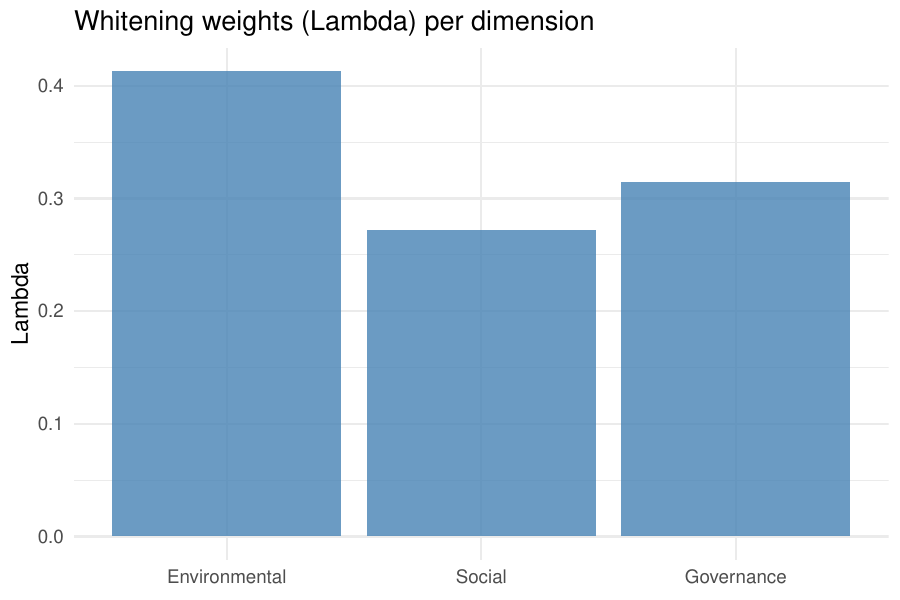}
    \caption{Weights of single pillars obtained through the whitening process. The Environmental score yields the highest weight, followed by the Governance pillar and the Social pillar.}
    \label{fig:weights}
\end{figure}

\subsection{Results and Comments}
\begin{table}[htbp]
\centering
\caption{Mean SAFE metrics for ESG pillars computed with 5-fold cross-validation across the two models (Linear Regression and Neural Network). Standard deviations are in parentheses.}
\resizebox{\textwidth}{!}{
\begin{tabular}{lcccccc}
\toprule
\multirow{2}{*}{\textbf{Metric}} & 
\multicolumn{2}{c}{\textbf{E.Sc}} & 
\multicolumn{2}{c}{\textbf{S.Sc}} & 
\multicolumn{2}{c}{\textbf{G.Sc}} \\
\cmidrule(lr){2-3} \cmidrule(lr){4-5} \cmidrule(lr){6-7}
 & \textbf{LM} & \textbf{NN} & \textbf{LM} & \textbf{NN} & \textbf{LM} & \textbf{NN} \\
\midrule
RGA & 0.641 (0.015) & 0.663 (0.046) & 0.528 (0.023) & 0.542 (0.051) & 0.657 (0.017) & 0.721 (0.030) \\
RGR & 0.973 (0.007) & 0.918 (0.050) & 0.908 (0.060) & 0.914 (0.023) & 0.912 (0.032) & 0.920 (0.047) \\
RGE\_Sector & 0.428 (0.078) & 0.375 (0.096) & 0.232 (0.060) & 0.337 (0.097) & 0.326 (0.098) & 0.265 (0.307) \\
RGE\_sTASS & 0.004 (0.004) & 0.190 (0.134) & 0.002 (0.002) & 0.343 (0.136) & 0.001 (0.001) & 0.325 (0.303) \\
RGE\_sSFND & 0.001 (0.000) & 0.120 (0.037) & 0.014 (0.009) & 0.270 (0.037) & 0.010 (0.005) & 0.305 (0.334) \\
RGE\_sTOVR & 0.004 (0.004) & 0.169 (0.036) & 0.014 (0.013) & 0.263 (0.072) & 0.005 (0.003) & 0.281 (0.266) \\
RGE\_sNINC & 0.002 (0.003) & 0.209 (0.115) & 0.011 (0.021) & 0.264 (0.082) & 0.005 (0.007) & 0.142 (0.099) \\
RGE\_sEBTD & 0.001 (0.001) & 0.287 (0.229) & 0.001 (0.002) & 0.360 (0.121) & 0.001 (0.001) & 0.197 (0.112) \\
\bottomrule
\end{tabular}}
\label{tab:univariate-safe}
\end{table}

The results of the univariate models, reported in Table \ref{tab:univariate-safe}, reveal several consistent patterns across the ESG pillars and model classes.
In terms of accuracy (RGA), neural networks generally outperform linear regressions, confirming their superior ability to capture nonlinear relationships between ESG scores and financial or sectoral characteristics. The lowest RGA values are observed for the Social pillar, suggesting that this dimension may be harder to explain through financial structure and sectoral composition alone. Among all specifications, the Governance pillar predicted through the neural network model achieves the highest accuracy (RGA = 0.721), whereas the Environmental pillar modelled through linear regression shows the most stable behaviuor, achieving the lowest standard deviation.

Regarding robustness (RGR), all models show a good performance, with scores around 0.9, indicating stable ranking predictions even under moderate perturbations of the predicted values. The linear model for the Environmental pillar shows the highest robustness (RGR = 0.973) and the smallest variability across folds, confirming its structural stability.

The interpretation of explainability (RGE) is more nuanced. The sectoral factor exhibits relatively high contributions across all models, confirming the relevance of industry effects in determining ESG performance. Conversely, financial variables show markedly lower RGE values in linear models but higher, yet more dispersed values in neural networks.

\begin{table}[htbp]
\centering
\caption{Multivariate SAFE metrics with whitening weights (5-fold cross-validation). Standard deviations in parentheses.}
\begin{tabular}{lcc}
\toprule
\textbf{Metric}  & \textbf{Linear Regression} & \textbf{Neural Network} \\
\midrule
RGA & 0.603 (0.010) & 0.641 (0.016) \\
RGR & 0.786 (0.027) & 0.840 (0.085) \\
RGE\_Sector & 0.333 (0.046) & 0.380 (0.115) \\
RGE\_sEBTD & 0.0009 (0.0008) & 0.291 (0.154) \\
RGE\_sNINC & 0.0056 (0.0069) & 0.257 (0.144) \\
RGE\_sSFND & 0.0075 (0.0054) & 0.292 (0.109) \\
RGE\_sTASS & 0.0030 (0.0015) & 0.287 (0.129) \\
RGE\_sTOVR & 0.0081 (0.0055) & 0.236 (0.047) \\
\bottomrule
\end{tabular}
\label{tab:multivariate-safe}
\end{table}

Multivariate results (Table \ref{tab:multivariate-safe}) reinforce the same general patterns we discussed for the univariate results. However, these models account for cross-pillar dependencies through the whitening transformation.
With respect to accuracy and robustness, the neural network again achieves higher mean RGA (0.641 vs. 0.603) and RGR (0.840 vs. 0.786) but also shows greater dispersion, consistent with its higher flexibility and sensitivity to data variation. This confirms the classical bias–variance trade-off observed in nonlinear models.

The explainability results reveal important structural differences. For the linear model, RGE values remain close to zero for most financial predictors, except the sector factor, indicating that the linear specification attributes limited marginal explanatory power to individual variables once joint correlations are accounted for. By contrast, the neural network assigns substantially higher RGE values across all predictors, probabily due to the fact that this model is able to capture explanatory importance more evenly across the feature space.
This pattern indicates that nonlinear interactions, absent in the linear case, play a key role in explaining ESG variability when the three pillars are modeled jointly.

\subsection{Shapley Value analysis}
To complement the RGE results and provide a model-agnostic measure of feature relevance, we compute Shapley values\cite{StrumbeljKononenko2014} for each feature using the Monte Carlo approximation method implemented in the \texttt{fastshap} package\cite{Greenwell2020}. This approach allows decomposing each model prediction into additive feature contributions, while at the same time ensuring fairness and consistency across different model structures.

Given a prediction model $f: \mathbb{R}^d \to \mathbb{R}$ trained on features $\mathcal{F} = \{1, \ldots, d\}$, the Shapley value $\phi_j$ for feature $j$ represents its average marginal contribution to the model output across all possible feature subsets:
\begin{equation}
\phi_j = \sum_{S \subseteq \mathcal{F} \setminus \{j\}} \frac{|S|!\,(d - |S| - 1)!}{d!} \left[ f(S \cup \{j\}) - f(S) \right].
\end{equation}

In practice, the exact enumeration of all $2^d$ coalitions is computationally costly. We therefore employ a Monte Carlo approximation, where for each observation $i$ and feature $j$, the marginal contribution is estimated over $M$ random feature orderings (here $M = 50$):
\begin{equation}
\hat{\phi}_j^{(i)} = \frac{1}{M} \sum_{m=1}^{M} \Big( f(x_{S_m \cup \{j\}}^{(i)}) - f(x_{S_m}^{(i)}) \Big).
\end{equation}

The global importance of feature $j$ is then obtained by averaging the absolute values of its local contributions across all test instances:
\begin{equation}
\bar{\phi}_j = \frac{1}{n} \sum_{i=1}^{n} |\hat{\phi}_j^{(i)}|.
\end{equation}

To facilitate interpretability and comparison, these mean absolute contributions are normalized to obtain percentage importances that sum to 100\%:
\begin{equation}
\text{Importance}_j = 
\frac{\bar{\phi}_j}{\sum_{k=1}^{d} \bar{\phi}_k} \times 100.
\end{equation}
These normalized quantities express the relative explanatory contribution of each predictor to the overall model output variance.

In the univariate specification, the procedure is applied to both the linear regression and neural network models, separately for each ESG pillar.
The analysis uses $K = 5$-fold cross-validation to ensure robustness of the estimated feature importances. For each fold:
\begin{enumerate}
    \item The training data are used to fit both models on the standardized predictors;
    \item Shapley values are computed on the held-out test data using $M=50$ random permutations per observation;
    \item Fold-specific results are aggregated by averaging the mean and standard deviation of Shapley values across folds.
\end{enumerate}

In the multivariate setting, separate Shapley analyses are conducted for each whitened dimension, obtaining dimension-specific importances \(\text{Importance}_j^{(k)}\).  
These are then aggregated using the absolute mean-based lambda weights obtained with the whitening transformation:
\begin{equation}
\text{Importance}_j^{(\text{multi})} = \sum_{k=1}^{3} \lambda_k \, \text{Importance}_j^{(k)}.
\end{equation}

As a result of this weighting, ESG dimensions with higher magnitude (i.e., larger contribution to total variance) show greater influence in the combined feature importance.

Results of this analysis are presented in Table \ref{tab:shapley} and illustrated in Figure \ref{fig:univariate} and Figure \ref{fig:multivariate}.
We interpret these results from three complementary perspectives:
(i) how each model behaves within each pillar (univariate setting);
(ii) how the pillars behave within each model;
(iii) how each model performs in the multivariate specification.

\begin{figure}[htbp]
  \centering
  \begin{subfigure}{0.69\textwidth}
    \includegraphics[trim={0cm 0cm 0cm 1.2cm},clip,width=\linewidth]{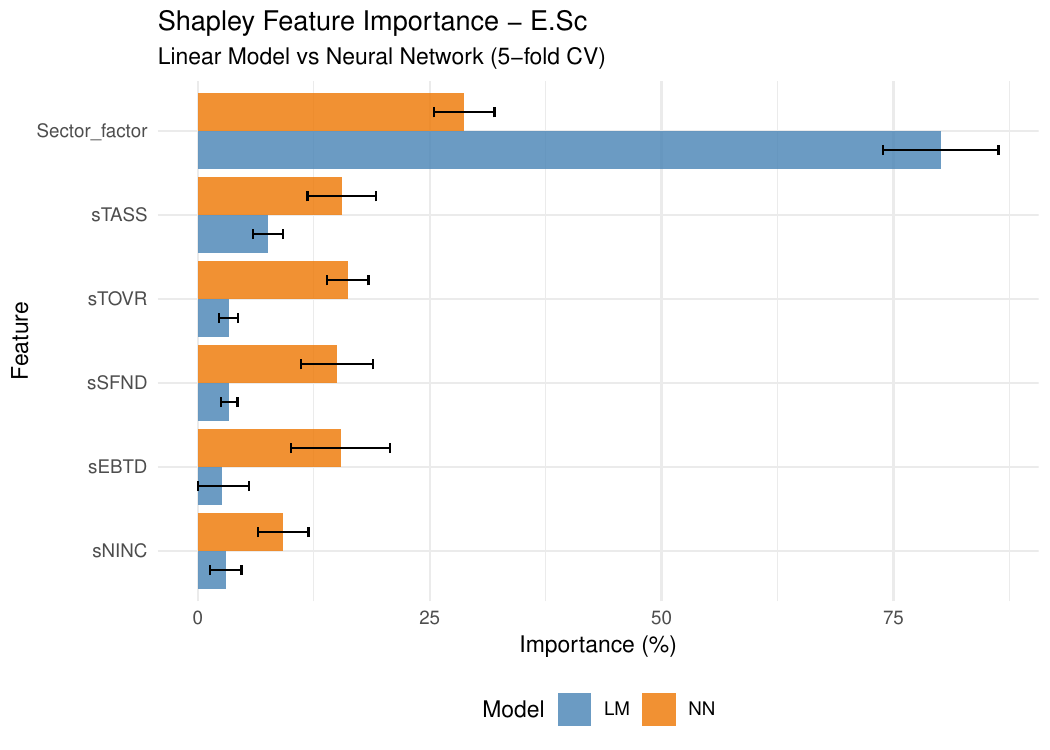}
    \caption{Environmental Pillar}\label{fig:env}
  \end{subfigure}
  \begin{subfigure}{0.69\textwidth}
    \includegraphics[trim={0cm 0cm 0cm 1.2cm},clip,width=\linewidth]{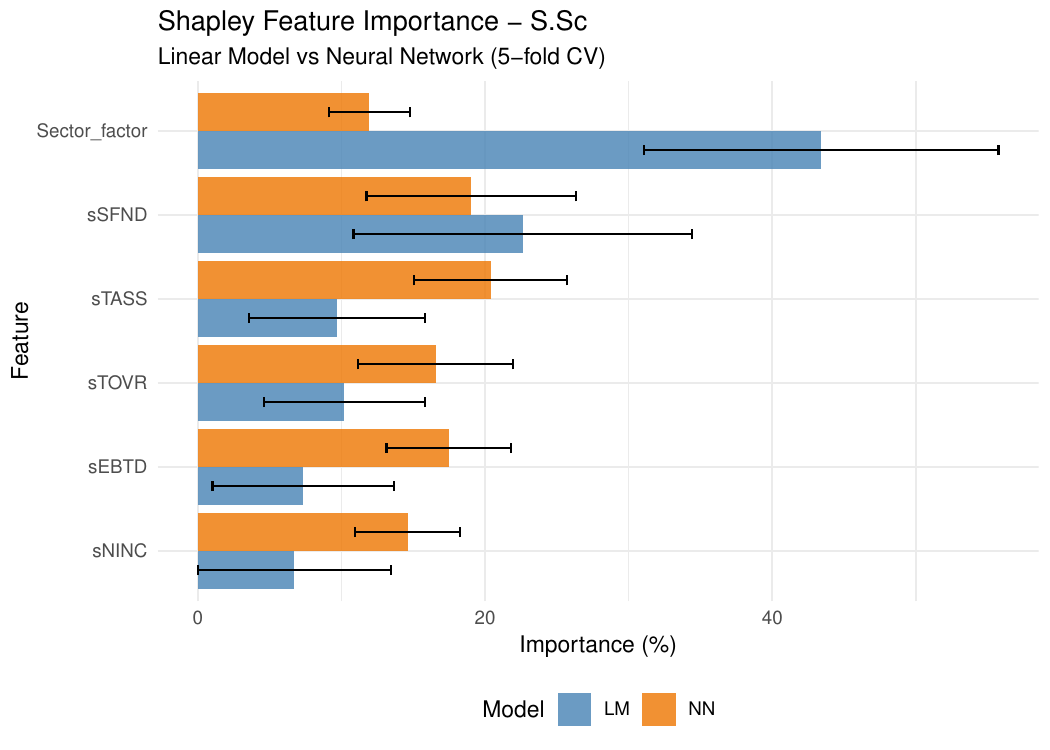}
    \caption{Social Pillar}\label{fig:soc}
  \end{subfigure}
  \begin{subfigure}{0.69\textwidth}
    \includegraphics[trim={0cm 0cm 0cm 1.2cm},clip,width=\linewidth]{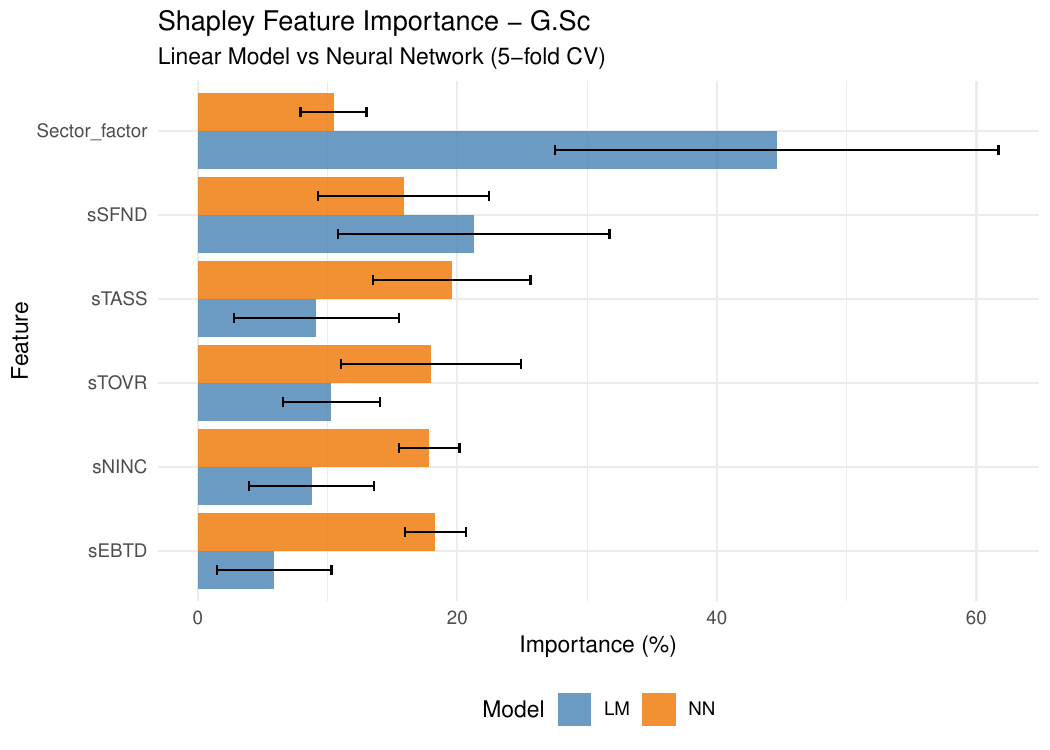}
    \caption{Governance Pillar}\label{fig:gov}
  \end{subfigure}
  \caption{Shapley-based feature importance for each pillar, comparing Linear Model ($LM$) and Neural Network ($NN$) under 5-fold cross-validation. Bars represent mean feature contributions to model predictions (in \%), with error bars indicating standard deviations across folds.}
  \label{fig:univariate}
\end{figure}

For the Environmental pillar (Figure \ref{fig:env}), the linear model (LM) assigns substantially higher importance to the sector factor compared to the neural network (NN). However, in both models, the sector factor remains the most influential variable. For LM, it also shows the largest standard deviation, suggesting that sector-related effects vary considerably across firms.
Among the remaining predictors, NN generally assigns higher importance scores than LM, which is consistent with the patterns previously observed in the RGE results. This likely reflects NN’s ability to distribute explainability more evenly across variables, while LM concentrates it on the sector component. In the NN, TASS and TOVR emerge as the most relevant financial variables after the sector factor.

For the Social pillar (Figure \ref{fig:soc}), the pattern is broadly similar. The LM again overemphasizes the sector factor, but here the SFND variable also gains relevance. Both features show higher variability, indicating less stable importance weights in LM.
In contrast, NN assigns relatively higher and more balanced importance to TASS and SFND, while the sector factor plays a more limited role. This might suggest that NN captures firm-level drivers of social performance.

Turning to the Governance pillar (Figure \ref{fig:gov}), LM continues to attribute dominant weight to the sector variable, with the highest standard deviation among all variables, indicating unstable contribution estimates. NN, instead, distributes relevance among multiple financial indicators.
Notably, TASS and EBTD receive greater importance under NN, implying that the nonlinear model may be better able to capture the role of profitability and firm structure in governance-related outcomes.

Considering all pillars within the LM, the sector variable consistently dominates, followed by SFND for the S.Sc and G.Sc pillars, and TASS for E.Sc. Standard deviations are systematically larger for sector-related contributions, reflecting their heterogeneous impact across firms.

For all pillars within the NN, the sector factor is most relevant only in the Environmental dimension. In the Social and Governance pillars, TASS and SFND emerge as leading predictors (with EBTD also important for Governance). This pattern suggests that the NN captures more nuanced, possibly nonlinear relationships, assigning weights that better reflect firm-specific and financial drivers of ESG performance.

Finally, in the multivariate specification (Figure \ref{fig:multivariate}), LM continues to emphasize the sector factor far more than NN, which again distributes explanatory power more evenly across predictors—mirroring the interpretability patterns seen in RGE. The most important variables for NN are TASS and SFND, while for LM the sector factor remains dominant, followed by TASS.
Overall, the results confirm that LM tends to overconcentrate importance on sectoral effects, whereas NN attributes more balanced and economically plausible importance across financial dimensions.

\begin{figure}
    \centering
    \includegraphics[trim={0cm 0cm 0cm 1.2cm},clip,width=0.8\linewidth]{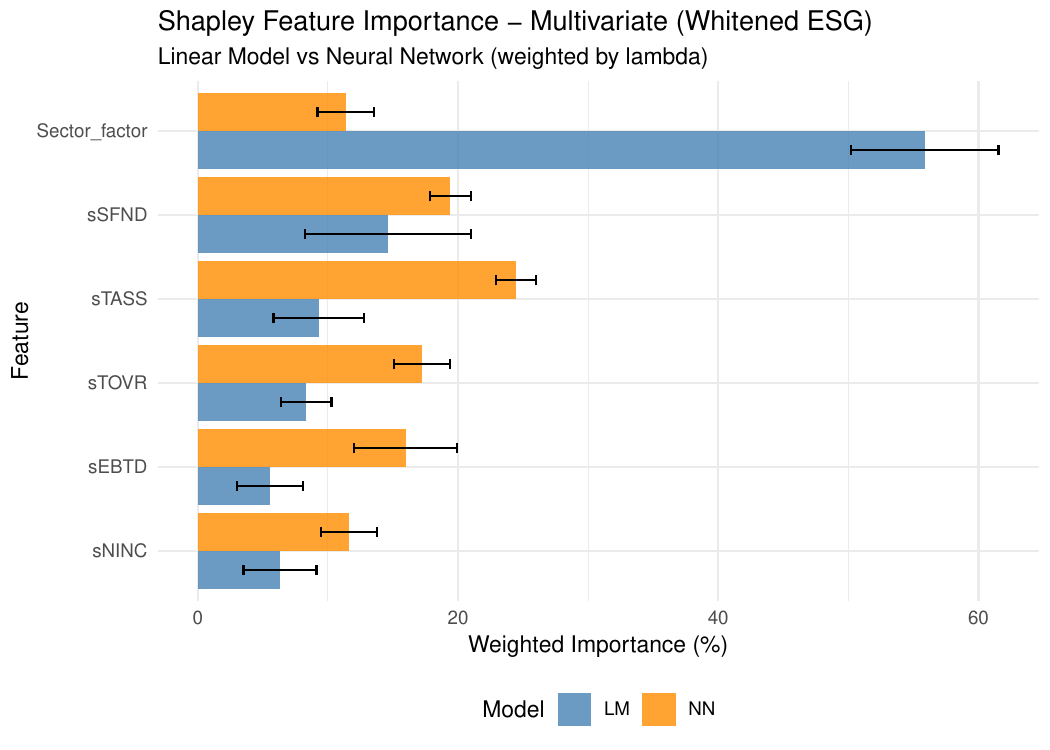}
    \caption{Shapley-based feature importance for multivariate models with whitened ESG components, comparing the Linear Model ($LM$) and Neural Network ($NN$). Feature importances (in \%) are averaged across folds and weighted by $\lambda$, with error bars representing standard deviations across folds.}
    \label{fig:multivariate}
\end{figure}

\begin{table}[htbp]
\centering
\caption{Mean and standard deviation of Shapley values and relative importances (Imp \%) across pillars and models.}
\label{tab:shapley}
\scriptsize
\begin{tabular}{llcccc}
\toprule
\multirow{2}{*}{\textbf{Pillar}} & 
\multirow{2}{*}{\textbf{Feature}} & 
\multicolumn{2}{c}{\textbf{Linear Regression}} &
\multicolumn{2}{c}{\textbf{Neural Network}} \\
\cmidrule(lr){3-4} \cmidrule(lr){5-6}
 &  & Shapley $\pm$ SD & Imp \% $\pm$ SD & Shapley $\pm$ SD & Imp \% $\pm$ SD \\
\midrule
\multicolumn{6}{c}{\textbf{Univariate Models}} \\
\midrule
\multirow{6}{*}{E.Sc} 
 & Sector & 0.0382 $\pm$ 0.0045 & 80.1 $\pm$ 6.2 & 0.0440 $\pm$ 0.0065 & 28.7 $\pm$ 3.3 \\
 & sTASS          & 0.0037 $\pm$ 0.0011 & 7.54 $\pm$ 1.63 & 0.0254 $\pm$ 0.0130 & 15.5 $\pm$ 3.7 \\
 & sSFND          & 0.0016 $\pm$ 0.0005 & 3.40 $\pm$ 0.88 & 0.0228 $\pm$ 0.0055 & 15.0 $\pm$ 3.9 \\
 & sTOVR          & 0.0016 $\pm$ 0.0005 & 3.33 $\pm$ 1.04 & 0.0258 $\pm$ 0.0101 & 16.2 $\pm$ 2.2 \\
 & sNINC          & 0.0015 $\pm$ 0.0009 & 3.00 $\pm$ 1.72 & 0.0146 $\pm$ 0.0056 & 9.23 $\pm$ 2.7 \\
 & sEBTD          & 0.0013 $\pm$ 0.0014 & 2.64 $\pm$ 2.84 & 0.0238 $\pm$ 0.0085 & 15.4 $\pm$ 5.3 \\
\midrule
\multirow{6}{*}{S.Sc} 
 & Sector & 0.0183 $\pm$ 0.0025 & 43.4 $\pm$ 12.3 & 0.0267 $\pm$ 0.0057 & 11.9 $\pm$ 2.8 \\
 & sSFND          & 0.0113 $\pm$ 0.0079 & 22.6 $\pm$ 11.8 & 0.0472 $\pm$ 0.0281 & 19.0 $\pm$ 7.3 \\
 & sTOVR          & 0.0045 $\pm$ 0.0028 & 10.2 $\pm$ 5.6 & 0.0377 $\pm$ 0.0153 & 16.6 $\pm$ 5.4 \\
 & sTASS          & 0.0048 $\pm$ 0.0037 & 9.68 $\pm$ 6.1 & 0.0482 $\pm$ 0.0215 & 20.4 $\pm$ 5.3 \\
 & sEBTD          & 0.0038 $\pm$ 0.0040 & 7.35 $\pm$ 6.3 & 0.0411 $\pm$ 0.0175 & 17.5 $\pm$ 4.3 \\
 & sNINC          & 0.0023 $\pm$ 0.0018 & 6.71 $\pm$ 6.8 & 0.0358 $\pm$ 0.0178 & 14.6 $\pm$ 3.7 \\
\midrule
\multirow{6}{*}{G.Sc} 
 & Sector & 0.0186 $\pm$ 0.0023 & 44.6 $\pm$ 17.1 & 0.152  $\pm$ 0.227  & 10.5 $\pm$ 2.5 \\
 & sSFND          & 0.0109 $\pm$ 0.0085 & 21.3 $\pm$ 10.5 & 0.336  $\pm$ 0.574  & 15.9 $\pm$ 6.6 \\
 & sTOVR          & 0.0052 $\pm$ 0.0038 & 10.3 $\pm$ 3.8  & 0.302  $\pm$ 0.474  & 18.0 $\pm$ 7.0 \\
 & sTASS          & 0.0055 $\pm$ 0.0055 & 9.14 $\pm$ 6.3  & 0.326  $\pm$ 0.563  & 19.6 $\pm$ 6.1 \\
 & sNINC          & 0.0046 $\pm$ 0.0044 & 8.78 $\pm$ 4.8  & 0.263  $\pm$ 0.392  & 17.8 $\pm$ 2.3 \\
 & sEBTD          & 0.0034 $\pm$ 0.0038 & 5.90 $\pm$ 4.4  & 0.343  $\pm$ 0.594  & 18.3 $\pm$ 2.3 \\
\midrule
\multicolumn{6}{c}{\textbf{Multivariate Models}} \\
\midrule
\multirow{6}{*}{} 
 & Sector & 0.0250 $\pm$ 0.0026 & 55.9 $\pm$ 5.7 & 0.0717 $\pm$ 0.0230 & 11.4 $\pm$ 2.2 \\
 & sSFND          & 0.0072 $\pm$ 0.0046 & 14.6 $\pm$ 6.4 & 0.282  $\pm$ 0.228  & 19.4 $\pm$ 1.6 \\
 & sTASS          & 0.0047 $\pm$ 0.0024 & 9.29 $\pm$ 3.5 & 0.443  $\pm$ 0.324  & 24.4 $\pm$ 1.6 \\
 & sTOVR          & 0.0039 $\pm$ 0.0007 & 8.34 $\pm$ 1.9 & 0.237  $\pm$ 0.175  & 17.2 $\pm$ 2.1 \\
 & sNINC          & 0.0028 $\pm$ 0.0010 & 6.32 $\pm$ 2.8 & 0.106  $\pm$ 0.0672 & 11.6 $\pm$ 2.1 \\
 & sEBTD          & 0.0029 $\pm$ 0.0015 & 5.55 $\pm$ 2.5 & 0.166  $\pm$ 0.105  & 16.0 $\pm$ 3.9 \\
\bottomrule
\end{tabular}
\end{table}

\subsubsection{Spearman's Correlation analysis}
In order to check the coherence of explainability across models and ESG pillars, we compute pairwise rank correlations using the Spearman coefficient (denoted by $\rho$).\cite{Spearman1987}  
This analysis quantifies whether features considered most important by one model are similarly ranked by another, and whether the same patterns of relevance persist across pillars or between univariate and multivariate specifications.

For each combination of ESG pillar ($E.Sc$, $S.Sc$, $G.Sc$) and model type (LM or NN), features are ranked in descending order according to their mean Shapley importance percentage. The ranking assigns lower numerical values to more important features, with average ranking used for ties.



Given two ranking vectors $\mathrm{R}_A = (R_{A1}, \ldots, R_{Ad})$ and $\mathrm{R}_B = (R_{B1}, \ldots, R_{Bd})$ for $d$ common features, the Spearman correlation coefficient is defined as:

\begin{equation}
\rho = 1 - \frac{6 \sum_{j=1}^{d} (R_{Aj} - R_{Bj})^2}{d (d^2 - 1)}.
\end{equation}

A value of $\rho = 1$ indicates perfect concordance (identical rankings), $\rho = -1$ perfect inversion (one ranking is the exact reverse of the other), and $\rho = 0$ no association between rankings.

Results are presented in Figure~\ref{fig:corrplot}.
The interpretation can be structured, like in the previous section, along three complementary perspectives: (i) how each model behaves within each pillar in the univariate setting; (ii) how the pillars behave within each model; (iii) how each model performs in the multivariate specification.

\begin{figure}
    \centering
    \includegraphics[trim={0cm 0cm 0cm 1cm},clip, width=0.9\linewidth]{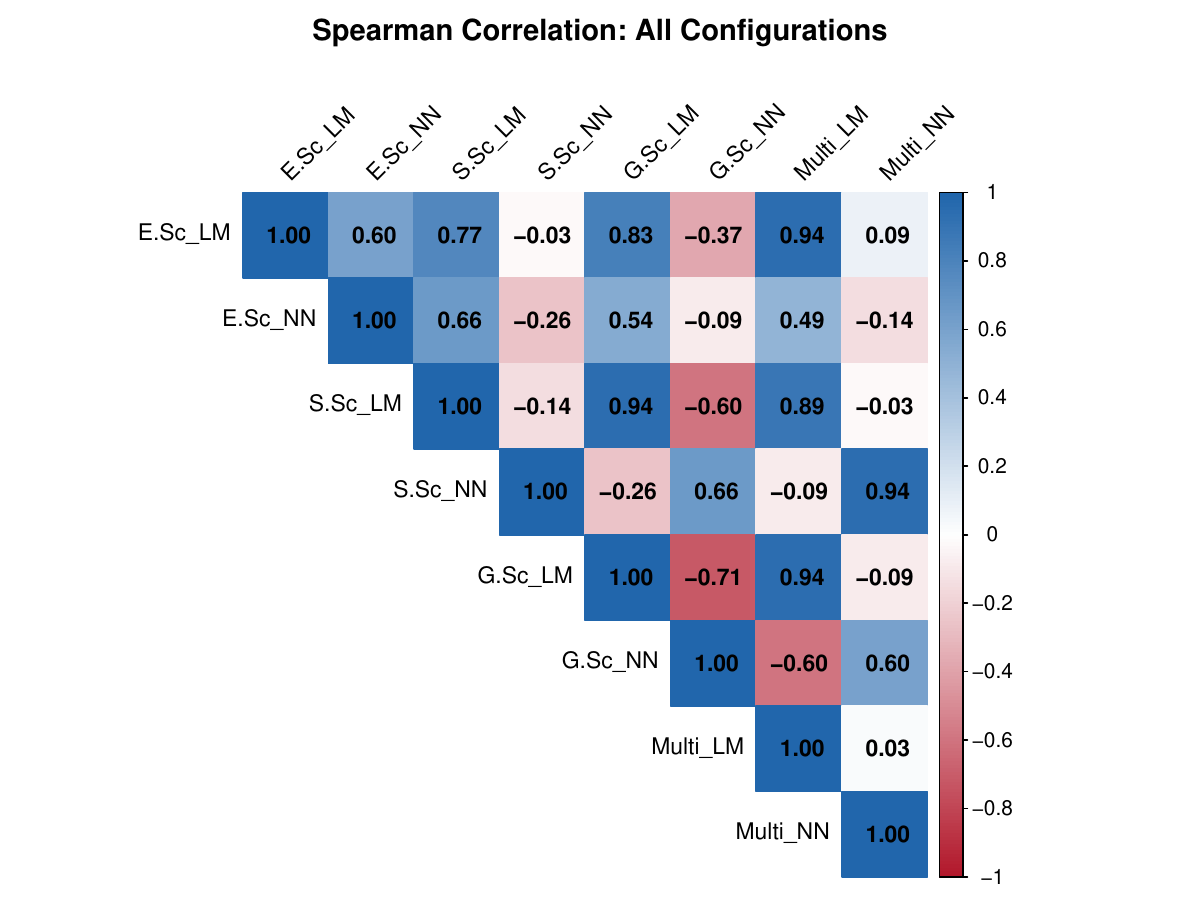}
    \caption{Spearman correlation matrix of feature importance rankings derived from Shapley values across univariate ($E.Sc$, $S.Sc$, $G.Sc$) and multivariate whitened models ($Multi$), for both Linear Model ($LM$) and Neural Network ($NN$) specifications.}
    \label{fig:corrplot}
\end{figure}

Rank correlations between the Linear Model (LM) and the Neural Network (NN) are generally low, confirming that the two models produce distinct patterns of variable importance. Specifically, correlations are moderate for the Environmental pillar ($\rho = 0.6$), small and negative for the Social pillar ($\rho = -0.143$), and negative and relatively strong for the Governance pillar ($\rho = -0.714$). This divergence is consistent with expectations: LM captures strictly linear relationships, while NN incorporates nonlinear interactions and redistributes importance weights across predictors accordingly. Because of this, explainability is different across models.

When keeping the model constant and comparing pillars, LM displays a high degree of internal consistency ($\rho \approx 0.8$ and above). The strongest association emerges between the Social and Governance pillars ($\rho = 0.943$), indicating that LM identifies relatively stable patterns of feature importance across ESG dimensions.
On the other hand, the NN model exhibits weak or even negative correlations across pillars, with only a modest positive association between the Social and Governance pillars ($\rho = 0.657$). This might indicate that the nonlinear model adapts its attribution patterns in a more flexible way across diverse contexts, possibly at the cost of stability and interpretability.

Finally, comparing univariate and multivariate rankings we can notice that LM maintains strong coherence across specifications ($\rho \approx 0.9$ for all pillars). This pattern suggests that variable importance in LM remains stable even when predictors interact jointly.
For NN, correlations with the multivariate configuration are generally weaker, except for the Social pillar, which displays a strong alignment with the multivariate model ($\rho = 0.943$). This again reflects NN’s flexibility in reallocating feature relevance across settings, consistent with its nonlinear structure.

\section*{Acknowledgment}

This work has been written within the activities of GNCS and GNFM groups of INdAM (Italian National Institute of High Mathematics). P.G.  has been funded by the European Union - NextGenerationEU, in the framework of the GRINS- Growing Resilient, INclusive and Sustainable (GRINS PE00000018). 
The work of A. E. Bernardelli has been conducted during and with the support of the Italian national inter-university PhD course in Sustainable Development and Climate change (link: www.phd-sdc.it)", Cycle XL, with the support of a scholarship co-financed under Ministerial Decree no. 630 of 28th March 2024 (Investment 3.3 – Innovative PhDs), based on the National Recovery and Resilience Plan (NRRP) - funded by the European Union - NextGenerationEU - Mission 4 "Education and Research", Component 2 "From Research to Business", and by Noto Sondaggi Srl.


\bibliographystyle{plain}
\bibliography{biblio}

\begin{thebibliography}{10}

\bibitem{Agarwal2023BiasLending}
S.~Agarwal, C.~Muckley, and P.~Neelakantan.
\newblock Countering racial discrimination in algorithmic lending: a case for model agnostic interpretation methods.
\newblock {\em Economics Letters}, 226, 2023.

\bibitem{Aldasoro2022CyberRisk}
I.~Aldasoro, L.~Gambacorta, P.~Giudici, and T.~Leach.
\newblock The drivers of cyber risk.
\newblock {\em Journal of Financial Stability}, 60:100989, 2022.

\bibitem{altschuler2019massively}
J.~Altschuler, F.~Bach, A.~Rudi, and J.~Niles-Weed.
\newblock Massively scalable {S}inkhorn distances via the {N}ystr{\"o}m method.
\newblock {\em Advances in neural information processing systems}, 32, 2019.

\bibitem{anshelevich2021ordinal}
E.~Anshelevich and W.~Zhu.
\newblock Ordinal approximation for social choice, matching, and facility location problems given candidate positions.
\newblock {\em ACM Transactions on Economics and Computation (TEAC)}, 9(2):1--24, 2021.

\bibitem{arnold2014pareto}
B.C. Arnold.
\newblock Pareto distribution.
\newblock {\em Wiley StatsRef: Statistics Reference Online}, pages 1--10, 2014.

\bibitem{auricchio2018computing}
G.~Auricchio, F.~Bassetti, S.~Gualandi, and M.~Veneroni.
\newblock Computing {K}antorovich-{W}asserstein distances on $ d $-dimensional histograms using $(d+ 1) $-partite graphs.
\newblock {\em Advances in Neural Information Processing Systems}, 31, 2018.

\bibitem{auricchio2019computing}
G.~Auricchio, F.~Bassetti, S.~Gualandi, and M.~Veneroni.
\newblock Computing {W}asserstein barycenters via linear programming.
\newblock In {\em International Conference on Integration of Constraint Programming, Artificial Intelligence, and Operations Research}, pages 355--363. Springer, 2019.

\bibitem{Auricchio2025MultivariateGini}
G.~Auricchio, G.~Brigati, P.~Giudici, and G.~Toscani.
\newblock Multivariate {G}ini-type discrepancies.
\newblock {\em Mathematical Models and Methods in Applied Sciences}, 35(5):1267--1296, 2025.

\bibitem{auricchio2025kinetic}
G.~Auricchio, G.~Brigati, P.~Giudici, and G.~Toscani.
\newblock From kinetic theory to {AI}: a rediscovery of high-dimensional divergences and their properties.
\newblock {\em Mathematical Models and Methods in Applied Sciences (Accepted, to appear)}, 2026.

\bibitem{auricchio2020equivalence}
G.~Auricchio, A.~Codegoni, S.~Gualandi, G.~Toscani, and M.~Veneroni.
\newblock The equivalence of {F}ourier-based and {W}asserstein metrics on imaging problems.
\newblock {\em Rendiconti Lincei}, 31(3):627--649, 2020.

\bibitem{auricchio2023fourier}
G.~Auricchio, A.~Codegoni, S.~Gualandi, L.~Zambon, et~al.
\newblock The {F}ourier discrepancy function.
\newblock {\em Commun. Math. Sci.}, 21(3):627--639, 2023.

\bibitem{Auricchio2025WhiteningGini}
G.~Auricchio, P.~Giudici, and G.~Toscani.
\newblock Extending the {G}ini index to higher dimensions via whitening processes.
\newblock {\em Rendiconti Lincei}, 35(3):511--528, 2025.

\bibitem{auricchio2025energy}
G.~Auricchio and G.~Toscani.
\newblock Energy distance and evolution problems: a promising tool for kinetic equations.
\newblock {\em arXiv preprint arXiv:2510.09123}, 2025.

\bibitem{auricchio2024extended}
G.~Auricchio, J.~Zhang, and M.~Zhang.
\newblock Extended ranking mechanisms for the m-capacitated facility location problem in bayesian mechanism design.
\newblock In {\em Proceedings of the 23rd International Conference on Autonomous Agents and Multiagent Systems}, pages 87--95, 2024.

\bibitem{Babaei2022ExplainableFintech}
G.~Babaei, P.~Giudici, and E.~Raffinetti.
\newblock Explainable fintech lending.
\newblock {\em Journal of Economics and Business}, 2022.

\bibitem{Babaei2025SAFEAI}
G.~Babaei, P.~Giudici, and E.~Raffinetti.
\newblock A rank graduation box for {SAFE AI}.
\newblock {\em Expert Systems with Applications}, 259:125239, 2025.

\bibitem{bellemare2017cramer}
M.G. Bellemare, I.~Danihelka, W.~Dabney, S.~Mohamed, B.~Lakshminarayanan, S.~Hoyer, and R.~Munos.
\newblock The {Cr}am{\'e}r distance as a solution to biased {W}asserstein gradients.
\newblock {\em arXiv preprint arXiv:1705.10743}, 2017.

\bibitem{bellomo2024life}
N.~Bellomo, M.~Dolfin, and J.~Liao.
\newblock Life and self-organization on the way to artificial intelligence for collective dynamics.
\newblock {\em Physics of Life Reviews}, 51:1--8, 2024.

\bibitem{bellomo2024herbert}
N.~Bellomo and M.~Egidi.
\newblock From {H}erbert {A}. {S}imon’s legacy to the evolutionary artificial world with heterogeneous collective behaviors.
\newblock {\em Mathematical Models and Methods in Applied Sciences}, 34(01):145--180, 2024.

\bibitem{Bengio2025}
Y.~Bengio et~al.
\newblock The {S}ingapore consensus on global {AI} safety.
\newblock {\em arXiv preprint arXiv:2506.20702}, 2025.

\bibitem{bouacida2021vulnerabilities}
N.~Bouacida and P.~Mohapatra.
\newblock Vulnerabilities in federated learning.
\newblock {\em IEEe Access}, 9:63229--63249, 2021.

\bibitem{Bracke2019Explainability}
P.~Bracke, A.~Datta, C.~Jung, and S.~Sen.
\newblock Machine learning explainability in finance: an application to default risk analysis.
\newblock Working Paper 816, Bank of England, 2019.

\bibitem{carrillo2007contractive}
J.A. Carrillo and G.~Toscani.
\newblock Contractive probability metrics and asymptotic behavior of dissipative kinetic equations.
\newblock {\em Riv. Mat. Univ. Parma}, 6(7):75--198, 2007.

\bibitem{Chen2024FairnessRatings}
Y.~Chen, P.~Giudici, L.~Kailang, and E.~Raffinetti.
\newblock Measuring fairness in credit ratings.
\newblock {\em Expert Systems with Applications}, 258:125184, 2024.

\bibitem{cheng2008neural}
J.~Cheng, Z.~Wang, and G.~Pollastri.
\newblock A neural network approach to ordinal regression.
\newblock In {\em 2008 IEEE international joint conference on neural networks (IEEE world congress on computational intelligence)}, pages 1279--1284. IEEE, 2008.

\bibitem{cuturi2013sinkhorn}
M.~Cuturi.
\newblock Sinkhorn distances: Lightspeed computation of optimal transport.
\newblock {\em Advances in neural information processing systems}, 26, 2013.

\bibitem{decancq2012inequality}
K.~Decancq and M.A. Lugo.
\newblock Inequality of wellbeing: A multidimensional approach.
\newblock {\em Economica}, 79(316):721--746, 2012.

\bibitem{EUAIAct2022}
{European Commission}.
\newblock Artificial {I}ntelligence {A}ct, 2022.

\bibitem{frank2001simple}
E.~Frank and M.~Hall.
\newblock A simple approach to ordinal classification.
\newblock In {\em European conference on machine learning}, pages 145--156. Springer, 2001.

\bibitem{frogner2015learning}
C.~Frogner, C.~Zhang, H.~Mobahi, M.~Araya, and T.A. Poggio.
\newblock Learning with a {W}asserstein loss.
\newblock {\em Advances in neural information processing systems}, 28, 2015.

\bibitem{gabetta1995metrics}
E.~Gabetta, G.~Toscani, and B.~Wennberg.
\newblock Metrics for probability distributions and the trend to equilibrium for solutions of the {B}oltzmann equation.
\newblock {\em Journal of statistical physics}, 81(5):901--934, 1995.

\bibitem{Gini1914Concentrazione}
C.~Gini.
\newblock Sulla misura della concentrazione e della variabilit\`a dei caratteri.
\newblock {\em Atti del Reale Istituto Veneto di Scienze, Lettere ed Arti}, 73:1203--1248, 1914.

\bibitem{Gini1921InequalityIncome}
C.~Gini.
\newblock Measurement of inequality of incomes.
\newblock {\em The Economic Journal}, 31:124--126, 1921.

\bibitem{GiudiciRaffinetti2024RGA}
P.~Giudici and E.~Raffinetti.
\newblock {RGA}: a unified measure of predictive accuracy.
\newblock {\em Advances in Data Analysis and Classification}, 2024.

\bibitem{Gneiting2011}
T.~Gneiting.
\newblock Making and evaluating point forecasts.
\newblock {\em Journal of the American Statistical Association}, 106(494):746--762, 2011.

\bibitem{Greenwell2020}
Brandon Greenwell and Maintainer~Brandon Greenwell.
\newblock Package ‘fastshap’.
\newblock {\em URL: https://CRAN. R-project. org/package= fastshap R package version 0.0}, 7, 2020.

\bibitem{gutierrez2015ordinal}
P.A. Guti{\'e}rrez, M.~Perez-Ortiz, J.~Sanchez-Monedero, F.~Fernandez-Navarro, and C.~Hervas-Martinez.
\newblock Ordinal regression methods: survey and experimental study.
\newblock {\em IEEE Transactions on Knowledge and Data Engineering}, 28(1):127--146, 2015.

\bibitem{gutierrez2014ordinal}
P.A. Guti{\'e}rrez, P.~Ti{\v{n}}o, and C.~Herv{\'a}s-Mart{\'\i}nez.
\newblock Ordinal regression neural networks based on concentric hyperspheres.
\newblock {\em Neural networks}, 59:51--60, 2014.

\bibitem{HandTill2001MulticlassAUC}
D.~Hand and R.J. Till.
\newblock A simple generalisation of the area under the {ROC} curve for multiple classification problems.
\newblock {\em Machine Learning}, 45:171--186, 2001.

\bibitem{hanjoul1987facility}
P.~Hanjoul and D.~Peeters.
\newblock A facility location problem with clients' preference orderings.
\newblock {\em Regional Science and Urban Economics}, 17(3):451--473, 1987.

\bibitem{Hurley2009SparsityMeasures}
N.~Hurley and S.~Rickard.
\newblock Comparing measures of sparsity.
\newblock {\em IEEE Transactions on Information Theory}, 55(10):4723--4741, 2009.

\bibitem{Kessy2018Whitening}
A.~Kessy, A.~Lewin, and K.~Strimmer.
\newblock Optimal whitening and decorrelation.
\newblock {\em The American Statistician}, 72(4):309--314, 2018.

\bibitem{koshevoy1996lorenz}
G.A. Koshevoy and K.~Mosler.
\newblock The {L}orenz zonoid of a multivariate distribution.
\newblock {\em Journal of the American Statistical Association}, 91(434):873--882, 1996.

\bibitem{koshevoy1997multivariate}
G.A. Koshevoy and K.~Mosler.
\newblock Multivariate {G}ini indices.
\newblock {\em Journal of Multivariate Analysis}, 60(2):252--276, 1997.

\bibitem{kullback1951information}
S.~Kullback and R.A. Leibler.
\newblock On information and sufficiency.
\newblock {\em The annals of mathematical statistics}, 22(1):79--86, 1951.

\bibitem{LiZhang1998Sphering}
G.~Li and J.~Zhang.
\newblock Sphering and its properties.
\newblock {\em Sankhy\={a}: The Indian Journal of Statistics, Series A}, pages 119--133, 1998.

\bibitem{Li2020GCNPrivacy}
X.~Li.
\newblock Graph convolutional networks for privacy metrics in online social networks.
\newblock {\em Social Networks in Applied Sciences}, 10, 2020.

\bibitem{LiuTerzi2010PrivacyScores}
K.~Liu and E.~Terzi.
\newblock A framework for computing the privacy scores of users in online social networks.
\newblock {\em ACM Transactions on Knowledge Discovery from Data}, 5:6, 2010.

\bibitem{liu2018constrained}
Y.~Liu, A.W.K. Kong, and C.K. Goh.
\newblock A constrained deep neural network for ordinal regression.
\newblock In {\em Proceedings of the IEEE conference on computer vision and pattern recognition}, pages 831--839, 2018.

\bibitem{Lorenz1905Concentration}
M.O. Lorenz.
\newblock Methods of measuring the concentration of wealth.
\newblock {\em Publications of the American Statistical Association}, 9:209--219, 1905.

\bibitem{Lundberg2017SHAP}
S.M. Lundberg and S.~Lee.
\newblock A unified approach to interpreting model predictions.
\newblock In {\em Advances in Neural Information Processing Systems}, volume~30, pages 4768--4777, 2017.

\bibitem{mitchell2022sampling}
R.~Mitchell, J.~Cooper, E.~Frank, and G.~Holmes.
\newblock Sampling permutations for shapley value estimation.
\newblock {\em Journal of Machine Learning Research}, 23(43):1--46, 2022.

\bibitem{Mitchell2021AlgorithmicFairnessReview}
S.~Mitchell, E.~Potash, S.~Barocas, A.~D’Amour, and K.~Lum.
\newblock Algorithmic fairness: choices, assumptions and definitions.
\newblock {\em Annual Review of Statistics and Its Application}, 8:141--163, 2021.

\bibitem{muthukumar2021classification}
V.~Muthukumar, A.~Narang, V.~Subramanian, M.~Belkin, D.~Hsu, and A.~Sahai.
\newblock Classification vs regression in overparameterized regimes: Does the loss function matter?
\newblock {\em Journal of Machine Learning Research}, 22(222):1--69, 2021.

\bibitem{NIST2023AIRMF}
{NIST}.
\newblock {AI} risk management framework, 2023.

\bibitem{OECD2022ClassificationAISystems}
{OECD}.
\newblock Framework for the classification of {AI} systems, 2022.

\bibitem{Passach2022FairnessSurvey}
D.~Passach and E.~Shmueli.
\newblock A review on fairness in machine learning.
\newblock {\em ACM Computing Surveys}, 55(3):1--44, 2022.

\bibitem{Pietra1915RelazioniIndici}
G.~Pietra.
\newblock {\em Delle relazioni tra gli indici di variabilit\`a}.
\newblock C. Ferrari, Pavia, 1915.

\bibitem{rabin2011wasserstein}
J.~Rabin and G.~Peyr{\'e}.
\newblock Wasserstein regularization of imaging problem.
\newblock In {\em 2011 18th IEEE international conference on image processing}, pages 1541--1544. IEEE, 2011.

\bibitem{Ribeiro2016WhyTrustYou}
M.T. Ribeiro, S.~Singh, and C.~Guestrin.
\newblock Why should {I} trust you? {E}xplaining the predictions of any classifiers.
\newblock In {\em Proceedings of the 22nd ACM SIGKDD International Conference on Knowledge Discovery and Data Mining}, pages 1135--1144, 2016.

\bibitem{rizzo2016energy}
M.L. Rizzo and G.J. Sz{\'e}kely.
\newblock Energy distance.
\newblock {\em Wiley interdisciplinary reviews: Computational statistics}, 8(1):27--38, 2016.

\bibitem{rozemberczki2022shapley}
B.~Rozemberczki, L.~Watson, P.~Bayer, H-T Yang, O.~Kiss, S.~Nilsson, and R.~Sarkar.
\newblock The {Sh}apley value in machine learning.
\newblock In {\em The 31st International Joint Conference on Artificial Intelligence and the 25th European Conference on Artificial Intelligence}, pages 5572--5579. International Joint Conferences on Artificial Intelligence Organization, 2022.

\bibitem{scagliotti2025normalizing}
A.~Scagliotti and S.~Farinelli.
\newblock Normalizing flows as approximations of optimal transport maps via linear-control neural {ODE}s.
\newblock {\em Nonlinear Analysis}, 257:113811, 2025.

\bibitem{song2025fedadmm}
Y.~Song, Z.~Wang, and E.~Zuazua.
\newblock Fed{ADMM}-{I}n{S}a: An inexact and self-adaptive {ADMM} for federated learning.
\newblock {\em Neural Networks}, 181:106772, 2025.

\bibitem{Spearman1987}
Charles Spearman.
\newblock The proof and measurement of association between two things.
\newblock {\em The American journal of psychology}, 100(3/4):441--471, 1987.

\bibitem{StrumbeljKononenko2014}
E.~{\v{S}}trumbelj and I.~Kononenko.
\newblock Explaining prediction models and individual predictions with feature contributions.
\newblock {\em Knowledge and Information Systems}, 41:647--665, 2014.

\bibitem{szekely2013energy}
G.J. Sz{\'e}kely and M.Ll Rizzo.
\newblock Energy statistics: A class of statistics based on distances.
\newblock {\em Journal of statistical planning and inference}, 143(8):1249--1272, 2013.

\bibitem{taguchi1972two}
T.~Taguchi.
\newblock On the two-dimensional concentration surface and extensions of concentration coefficient and {P}areto distribution to the two dimensional case—i: On an application of differential geometric methods to statistical analysis.
\newblock {\em Annals of the Institute of Statistical Mathematics}, 24(1):355--381, 1972.

\bibitem{tartavel2016wasserstein}
G.~Tartavel, G.~Peyr{\'e}, and Y.~Gousseau.
\newblock Wasserstein loss for image synthesis and restoration.
\newblock {\em SIAM Journal on Imaging Sciences}, 9(4):1726--1755, 2016.

\bibitem{villani2008optimal}
C.~Villani.
\newblock {\em Optimal transport: old and new}, volume 338.
\newblock Springer, 2008.

\bibitem{yang2019federated}
Q.~Yang, Y.~Liu, T.~Chen, and Y.~Tong.
\newblock Federated machine learning: Concept and applications.
\newblock {\em ACM Transactions on Intelligent Systems and Technology (TIST)}, 10(2):1--19, 2019.

\bibitem{ye2017fast}
J.~Ye, P.~Wu, J.Z. Wang, and J.~Li.
\newblock Fast discrete distribution clustering using {W}asserstein barycenter with sparse support.
\newblock {\em IEEE Transactions on Signal Processing}, 65(9):2317--2332, 2017.

\bibitem{yin2019understanding}
M.~Yin, Wortman V.J., and H.~Wallach.
\newblock Understanding the effect of accuracy on trust in machine learning models.
\newblock In {\em Proceedings of the 2019 chi conference on human factors in computing systems}, pages 1--12, 2019.

\bibitem{6467150}
L.~Zhang, L.~Zhang, X.~Mou, and D.~Zhang.
\newblock A comprehensive evaluation of full reference image quality assessment algorithms.
\newblock In {\em 2012 19th IEEE International Conference on Image Processing}, pages 1477--1480, 2012.

\bibitem{zhuang2022wasserstein}
Y.~Zhuang, X.~Chen, and Y.~Yang.
\newblock Wasserstein $ {K} $-means for clustering probability distributions.
\newblock {\em Advances in Neural Information Processing Systems}, 35:11382--11395, 2022.

\bibitem{zuazua2025machine}
E.~Zuazua.
\newblock Machine learning and control: Foundations, advances, and perspectives.
\newblock {\em arXiv preprint arXiv:2510.03303}, 2025.

\end{thebibliography}
\end{document}